\documentclass[12pt]{amsart}

\usepackage[margin=1.3in]{geometry}
\usepackage[utf8]{inputenc}
\usepackage[T1]{fontenc}

\usepackage{subfiles}
\usepackage{comment}
\usepackage{float}
\usepackage{marginnote}
\usepackage{euscript}
\usepackage[dvipsnames]{xcolor}
\usepackage{graphicx}
\usepackage{amssymb}
\usepackage{mathrsfs}
\usepackage{amsthm}
\usepackage{amsmath}
\usepackage{mathtools}
\usepackage[cal=cm]{mathalfa}
\usepackage{stmaryrd}
\usepackage{upgreek}
\usepackage{bbm}
\usepackage[breaklinks, hypertexnames=false, hyperfootnotes=false, colorlinks=true,linktocpage=true, allcolors=highlight]{hyperref}
\usepackage{cleveref}
\usepackage{url}
\usepackage{float}
\usepackage{todonotes} %\setuptodonotes{color=blue!30}
\usepackage[full]{textcomp}
\usepackage{tikz,tikz-cd,tikz-3dplot}
\usepackage{pgfplots}
\usetikzlibrary{calc}
\usetikzlibrary{fadings}
\usetikzlibrary{decorations.pathmorphing}
\usetikzlibrary{decorations.pathreplacing}
\usetikzlibrary{patterns}
\usetikzlibrary{arrows,shadows,positioning, calc, decorations.markings,
	hobby,quotes,angles,decorations.pathreplacing,intersections,shapes}
\usetikzlibrary{fillbetween,backgrounds}
\usepackage{xcolor}

\definecolor{highlight}{HTML}{0b5394}
\definecolor{faint}{HTML}{666666}

\usepackage{anyfontsize}
\usepackage[lining]{libertine}%
\usepackage{courier}
\usepackage[T1]{fontenc}
\usepackage[utf8]{inputenc}
\usepackage{csquotes}
\makeatletter
\renewcommand*\libertine@figurestyle{LF}
\makeatother
\usepackage[libertine,libaltvw,liby]{newtxmath}
\usepackage[final]{microtype}
\usepackage{array}
\newcolumntype{H}{>{\setbox0=\hbox\bgroup}c<{\egroup}@{}}
\makeatletter
\def\@tocline#1#2#3#4#5#6#7{\relax
	\ifnum #1>\c@tocdepth % then omit
	\else
	\par \addpenalty\@secpenalty\addvspace{#2}%
	\begingroup \hyphenpenalty\@M
	\@ifempty{#4}{%
		\@tempdima\csname r@tocindent\number#1\endcsname\relax
	}{%
		\@tempdima#4\relax
	}%
	\parindent\z@ \leftskip#3\relax \advance\leftskip\@tempdima\relax
	\rightskip\@pnumwidth plus4em \parfillskip-\@pnumwidth
	#5\leavevmode\hskip-\@tempdima
	\ifcase #1
	\or\or \hskip 1em \or \hskip 2em \else \hskip 3em \fi%
	#6\nobreak\relax
	\dotfill\hbox to\@pnumwidth{\@tocpagenum{#7}}\par
	\nobreak
	\endgroup
	\fi}
\makeatother

\tikzset{
	commutative diagrams/.cd,
	arrow style=tikz,
	diagrams={>=stealth}
}
\tikzset{
	arrow/.pic={\path[tips,every arrow/.try,->,>=#1] (0,0) -- +(0,4pt);},
	pics/arrow/.default={triangle 90}
}
\tikzset{->-/.style={decoration={
			markings,
			mark=at position .6 with {\arrow{latex}}},postaction={decorate}}
}
\tikzset{
	c/.style={every coordinate/.try}
}
\usepackage{enumerate}
\usepackage{enumitem}
\setlist[enumerate,1]{%
	label=(\roman*)	,itemsep=0.3em
	}
\setlist[itemize]{itemsep=0.3em}

\newcommand{\define}[1]{\textbf{#1}}

\newcommand{\mr}{\mathrm}
\newcommand{\mf}{\mathfrak}

\newcommand{\mgp}{\mathsf}

\DeclareMathOperator{\Aut}{Aut}

\DeclareMathOperator{\Tot}{Tot}
\DeclareMathOperator{\Rep}{Rep}
\DeclareMathOperator{\Exp}{Exp}
\DeclareMathOperator{\Log}{Log}

\newcommand{\vir}{\mathrm{virt}}

\newcommand{\actson}{\,\rotatebox[origin=c]{-90}{$\circlearrowright$}\,}

\usepackage{xifthen}
\newcommand{\Aaff}[1][]{%
	\ifthenelse{\isempty{#1}}%
	{\mathbb{C}}% if #1 is empty
	{
		\ifthenelse{\equal{#1}{1}}{\mathbb{C}}{\mathbb{C}^{#1}}
	}% if #1 is not empty
} 
\newcommand{\Gm}[1][]{%
\ifthenelse{\isempty{#1}}%
{\mathbb{C}^\times}% if #1 is empty
{(\mathbb{C}^\times)^{#1}}% if #1 is not empty
} 
\newcommand{\gpT}{\mgp{T}}

\newcommand{\modulifont}[1]{#1}
\newcommand{\Mbar}{\overline{\modulifont{M}}}

\newcommand{\invariantfont}[1]{\mathsf{#1}}
\newcommand{\GW}{\invariantfont{GW}}
\newcommand{\GWarg}[3]{\GW_{#3}(#1,#2)}
\newcommand{\GWdiscarg}[3]{\GW^\bullet_{#3}(#1,#2)}

\newcommand{\MTwoMod}{\invariantfont{M2}}

\newcommand{\ri}{\mathrm{i}}
\newcommand{\re}{\mathrm{e}}

\newcommand{\vertspace}{\hspace{0.2ex}|\hspace{0.2ex}}

\newcommand{\sd}{anti-diagonal}
\newcommand{\locsd}{locally anti-diagonal}
\newcommand{\MemInd}{membrane index}
\newcommand{\MemInds}{membrane indices}
\newcommand{\distdir}{distinct direction}
\newcommand{\frp}{framing parity}
\newcommand{\mxp}{mixing parity}
\newcommand{\frmxp}{framing and mixing parity}

\def\capmystringaux#1#2\relax{\uppercase{#1}\lowercase{#2}}

\newcommand{\hhh}{\mathsf{H}}

\newcommand{\HodgeLambda}[2][]{\Lambda_{#1}(#2)
}

\newcommand{\transpose}{\mathrm{t}}

\newcommand\isomrightarrow{\stackrel{\sim}{\smash{\longrightarrow}\rule{0pt}{0.4ex}}}

\makeatletter % changes the catcode of @ to 11
\def\makeCal#1{%
	\expandafter\newcommand\csname c#1\endcsname{\mathcal{#1}}}

\def\makeBB#1{%
	\expandafter\newcommand\csname b#1\endcsname{\mathbb{#1}}}

\def\makeFrak#1{%
	\expandafter\newcommand\csname f#1\endcsname{\mathfrak{#1}}}

\def\makeScr#1{%
	\expandafter\newcommand\csname s#1\endcsname{\mathscr{#1}}}

\def\makeSF#1{%
	\expandafter\newcommand\csname sf#1\endcsname{\mathsf{#1}}}

\count@=0
\loop
\advance\count@ 1
\edef\y{\@Alph\count@} % Alph is upper case
\expandafter\makeCal\y
\expandafter\makeBB\y
\expandafter\makeFrak\y
\expandafter\makeScr\y
\expandafter\makeSF\y
\ifnum\count@<26
\repeat

\theoremstyle{plain}
\newtheorem{thm}{Theorem}[section]
\newtheorem{lem}[thm]{Lemma}

\newtheorem{prop}[thm]{Proposition}
\newtheorem{conj}[thm]{Conjecture}

\newtheorem*{conj*}{Conjecture}

\newtheorem{cor}[thm]{Corollary}
\newtheorem*{cor*}{Corollary}

\theoremstyle{definition}
\newtheorem{defn}[thm]{Definition}
\newtheorem{rmk}[thm]{Remark}

\newtheorem{example}[thm]{Example}

\theoremstyle{plain}
\newtheorem{innercustomthm}{Theorem}
\newenvironment{customthm}[1]
{\renewcommand\theinnercustomthm{#1}\innercustomthm}
{\endinnercustomthm}

\theoremstyle{plain}

\theoremstyle{plain}
\newtheorem{innercustomconj}{Conjecture}
\newenvironment{customconj}[1]
{\renewcommand\theinnercustomconj{#1}\innercustomconj}
{\endinnercustomconj}

\theoremstyle{definition}

\theoremstyle{plain}

\crefname{equation}{Eq.}{Eqs.}
\crefname{eqnarray}{Eq.}{Eqs.}
\crefname{algo}{algorithm}{algorithms}
\crefname{conj}{conjecture}{conjectures}
\crefname{lem}{lemma}{lemmas}
\crefname{thm}{theorem}{theorems}
\crefname{claim}{claim}{claims}
\crefname{rmk}{remark}{remarks}
\crefname{prop}{proposition}{propositions}
\crefname{section}{section}{sections}
\crefname{appendix}{appendix}{appendices}
\crefname{cor}{corollary}{corollaries}
\crefname{figure}{figure}{figures}
\crefname{table}{table}{tables}
\crefname{example}{example}{examples}
\crefname{prob}{problem}{problems}
\crefname{assm}{assumption}{assumptions}
\crefname{defn}{definition}{definitions}
\crefname{speculation}{speculation}{speculations}
\crefname{construction}{construction}{constructions}
\crefname{innercustomthm}{theorem}{theorems}
\crefname{innercustomconj}{conjecture}{conjectures}
\crefname{innercustomassumption}{assumption}{Assumption}
\crefname{innerpracticalresult}{practical result}{practical results}

\usepackage[
		backend=biber,
		style=alphabetic,
		natbib=true,
		url=false,
		doi=false,
		eprint=true,
		isbn=false,
		maxcitenames=5,
		maxbibnames=4,
		maxalphanames=4,
		useprefix=true,
		giveninits=true
	]{biblatex}

\AtEveryBibitem{\clearlist{language}}

\renewbibmacro{in:}{%
	\ifboolexpr{%
		test {\ifentrytype{article}}%
		or
		test {\ifentrytype{inproceedings}}%
	}{}{\printtext{\bibstring{in}\intitlepunct}}%
}

\DeclareFieldFormat[article,inbook,incollection,inproceedings,patent,thesis,unpublished,techreport,misc,book]{title}{\mkbibquote{#1}}

\title[GW invariants and membrane indices of 5-folds via TV]{Gromov--Witten invariants and membrane indices\\of fivefolds via the topological vertex}
\author{Yannik Schuler}

\address{ETH Zurich, Department of Mathematics, Rämistr.~101, 8092 Zurich, Switzerland}
\email{yannik.schuler@math.ethz.ch}

\begin{document}

\begin{abstract}
	We conjecture the existence of almost integer invariants governing the all-genus equivariant Gromov--Witten theory of Calabi--Yau fivefolds with a torus action. We prove the conjecture for skeletal, \locsd\ torus actions by establishing a vertex formalism evaluating the Gromov--Witten invariants via the topological vertex of Aganagic, Klemm, Mari\~{n}o and Vafa. We apply the formalism in several examples.
\end{abstract}

\maketitle
\vspace*{-2em}
\setcounter{tocdepth}{1}
\hypersetup{bookmarksdepth = 3}
\tableofcontents

%-------------------------------------------------------------------------------
\vspace*{-2em}
\section*{Introduction}

\subsection{Gromov--Witten invariants and the membrane index}
In recent work of Brini and the author \cite{BS24:refined}, a conjecture was put forward relating the equivariant Gromov--Witten invariants of a Calabi--Yau fivefold to the so-called membrane index. Concretely, let $Z$ be a Calabi--Yau fivefold equipped with the action of a torus $\gpT$ fixing the holomorphic fiveform. The conjecture asserts that the series
\begin{equation*}
	\GWarg{Z}{\gpT}{\beta} \coloneqq \sum_{g\geq 0}\int_{[\Mbar_{g}(Z,\beta)]^{\vir}_{\gpT}}1
\end{equation*}
of equivariant Gromov--Witten invariants equals the \smash{$\hat{A}$}-genus of the mathematically yet-to-be-constructed moduli space of M2-branes on $Z$. While the Gromov--Witten series is a formal power series in the torus weights $\epsilon_i$, the \smash{$\hat{A}$}-genus admits a lift to K-theory under reasonable assumptions on the moduli space of M2-branes. This makes the latter a rational function in $\re^{\epsilon_i}$. Thus, at a numerical level, the conjecture predicts a lift of the Gromov--Witten series to a rational function.

\begin{customconj}{A}
	\label{conj: generalised GV}
	Let $Z$ be a Calabi--Yau fivefold with a Calabi--Yau $\gpT$-action. There exist rational functions
	\begin{equation*}
		\Omega_{\beta}(q_i) \in \bZ\big[\tfrac{1}{2}\big]\left[ q_{1}^{\pm 1/2}, \ldots, q_{\dim \gpT}^{\pm 1/2} , \, \left\{\frac{1}{1-\prod_i q_i^{n_i/2}}\right\}_{\boldsymbol{n}\in \bZ^{\dim \gpT}\setminus \{0\}} \,\right]
	\end{equation*}
	labelled by curve classes $\beta$ in $Z$ such that under the change of variables $q_i= \re^{\epsilon_i}$ we have
	\begin{equation}
		\label{eq: generalised GV}
		\GWarg{Z}{\gpT}{\beta} = \sum_{%\substack{k\in\bZ_{> 0}\\
				k \vertspace \beta%}
			} \frac{1}{k} \, \Omega_{\beta/k}(q_i^k) \,.
	\end{equation}
	Moreover, $\Omega_{\beta}$ has coefficients in $\bZ$ if the $\gpT$-action on $Z$ is skeletal.
\end{customconj}

We will refer to the invariant $\Omega_{\beta}$ as the \define{\MemInd} of $Z$ in curve class $\beta$. This index has earlier been investigated in relation to K-theoretic Donaldson--Thomas theory by Nekrasov and Okounkov \cite{NO14:membranes}.

In this paper we prove \Cref{conj: generalised GV} for a special class of geometries.

\begin{customthm}{B} (\Cref{thm: main part generalised GV curve class loc sd})
	\label{thm: generalised GV loc sd}
	\Cref{conj: generalised GV} holds if the torus action on $Z$ is skeletal and \locsd\ and curve classes are supported away from \sd\ strata.
\end{customthm}

We call a torus action skeletal if the number of fixed points and one-dimensional orbits is finite. We say that such an action is \locsd\ if the weight decomposition of the tangent space at every fixed point of $Z$ features two torus weights which are opposite to each other (see \Cref{defn: assumptions T action}). The assumption on curve classes ensures that stable maps do not interact with strata of the one-skeleton of $Z$ on which the torus acts with opposite weights.

\subsection{A vertex formalism for locally \sd\ torus actions}
\Cref{thm: generalised GV loc sd} is proven by a direct evaluation of the Gromov--Witten series.

\begin{customthm}{C} (\Cref{cor: vertex formalism curve class})
	\label{thm: vertex formalism}
	The disconnected Gromov--Witten invariants of a Calabi--Yau fivefold $Z$ with a Calabi--Yau, skeletal and \locsd\ action by a torus $\gpT$ in a curve class $\beta$ supported away from \sd\ strata are computed via the topological vertex of Aganagic--Klemm--Mari\~{n}o--Vafa \cite{AKMV05:TopVert}:
	\begin{equation}
		\label{eq: vertex formula}
		\GWdiscarg{Z}{\gpT}{\beta} = \sum_{\boldsymbol{\mu}} ~ \prod_{e \in E(\Gamma)} E(e,\boldsymbol{\mu}) \prod_{v \in V(\Gamma)} \cW(v,\boldsymbol{\mu})\,.
	\end{equation}
\end{customthm}

Here, the sum ranges over tuples of partitions decorating the half-edges of the one-skeleton $\Gamma$ of $Z$. Modulo a change of basis one may think of these partitions as indicating a contact order of relative stable maps mapping to a degeneration of the one-skeleton. Every edge $e$ of the one-skeleton is weighted by an explicit monomial $E(e,\boldsymbol{\mu})$ in the exponentiated weights $q_i = \re^{\epsilon_i}$ in formula \eqref{eq: vertex formula}. Vertices $v$ are weighted by the three-leg topological vertex $\cW(v,\boldsymbol{\mu})$ of Aganagic--Klemm--Mari\~{n}o--Vafa \cite{AKMV05:TopVert}. \Cref{thm: generalised GV loc sd} then follows from \Cref{thm: vertex formalism} since both edge and vertex weights are rational functions in \smash{$q_i^{1/2}$} with constrained poles.

%The fact that one encounters at most three non-trivially decorated legs is due to the assumption that stable maps may not be supported on \sd\ strata. This assumption rules out two of the five possible legs at each vertex.

The reader might be surprised to encounter the topological vertex in our vertex formula \eqref{eq: vertex formula} as it is an object usually appearing in the Gromov--Witten theory of threefolds as opposed to fivefolds.  This is explained as follows: By our assumption that the torus action is \locsd\ a fivefold vertex weight reduces to a threefold one with the $\gpT$-weight of the attached \sd\ stratum taking the role of the genus counting variable.

This last comment is probably best illustrated in the situation where $Z$ is the product of a toric Calabi--Yau threefold $X$ with $\Aaff[2]$. We assume that the threefold is acted on by its two-dimensional Calabi--Yau torus $\gpT'$ and $\Gm$ is acting anti-diagonally on the affine plane:
\begin{equation*}
	\gpT=\gpT'\times \Gm \qquad \actson \qquad Z= X \times \Aaff[2] \,.
\end{equation*}
In this situation the Gromov--Witten invariants of $Z$ simply coincide with the ones of $X$ with the equivariant parameter $\epsilon$ associated with the weight-one representation of $\Gm$ taking the role of the genus counting variable:
\begin{equation}
	\label{eq: GW X times C2 sd limit}
	\GWarg{Z}{\gpT}{\beta} = \sum_{g \geq 0} (-\epsilon^2)^{g-1} \int_{[\Mbar_{g}(X,\beta)]^{\vir}_{\gpT'}}1\,.
\end{equation}
This fact is a consequence of Mumford's relation for the Chern classes of the Hodge bundle \cite{Mu83}. See \cite[Sec.~2.4]{BS24:refined} for more details. Consistent with this dimensional reduction, our vertex formula \eqref{eq: vertex formula} specialises to the original topological vertex formula \cite{AKMV05:TopVert} formalised in Gromov--Witten theory by Li, Liu, Liu and Zhou \cite{LLLZ09:MathTopVert}. See \Cref{sec: TV limit} for details.

The key idea in the proof of \Cref{thm: vertex formalism} is that one can employ the same dimensional reduction trick \eqref{eq: GW X times C2 sd limit} locally at each vertex of the one-skeleton of $Z$ under the assumption that there are two tangent directions with opposite $\gpT$-weights. The same idea was recently pursued by Yu--Zong \cite{YZ26:1LegOrbVertRefGW} in the context of the one-leg orbifold vertex.

\subsection{Relation to Gopakumar--Vafa invariants}
Let us also quickly comment on the specialisation of \Cref{conj: generalised GV} to the product case $Z=X \times \Aaff[2]$ with the action by $\gpT'\times \Gm$ we just discussed. Here, we do not necessarily assume $X$ to be toric. We claim that in this setting the conjecture specialises to a weak version of the Gopakumar--Vafa integrality conjecture \cite{GV98:MthTopStrI} which was recently proven by Ionel--Parker and Doan--Ionel--Parker \cite{IP18:GV,DIW21:GVfinite}.

Indeed, suppose \Cref{conj: generalised GV} holds for $X \times \Aaff[2]$. Then since by the MNOP conjecture \cite{MNOP1,Par23:UnivCurveCount} the equivariant Gromov--Witten invariants of $X$ are just numbers independent of any $\gpT'$-weights, the functions $\Omega_\beta$ must satisfy
\begin{equation*}
	\Omega_{\beta} \in \bZ\big[\tfrac{1}{2}\big]\left[  q^{\pm 1/2} , \big\{\big(1-q^{n/2}\big)^{-1}\big\}_{n\in\bZ\setminus\{0\}}\right]
\end{equation*}
where $q=\re^{\epsilon}$. Now if we \textit{additionally assume} that $\Omega_{\beta}$ has integer coefficients, that it only features integer powers of $q$, and that $\Omega_{\beta}$ can have at worst a double pole at $q=1$ and no other poles but at zero and infinity, we can expand
\begin{equation*}
	\Omega_{\beta} = \sum_{g = 0}^{g_{\text{max}}} n_{g,\beta} \left(q^{1/2}-q^{-1/2}\right)^{2g-2}
\end{equation*}
with $n_{g,\beta}\in\bZ$. Here, we also used that the Gromov--Witten series is invariant under $\epsilon\mapsto -\epsilon$. In combination with \eqref{eq: GW X times C2 sd limit}, this exactly yields the statement of the Gopakumar--Vafa conjecture:
\begin{equation*}
	\sum_{g \geq 0} (-\epsilon^2)^{g-1} \int_{[\Mbar_{g}(X,\beta)]^{\vir}_{\gpT'}}1 = \sum_{k\vertspace\beta}\sum_{g = 0}^{g_{\text{max}}} \frac{n_{g,\beta/k}}{k} \left(q^{k/2}-q^{-k/2}\right)^{2g-2}\,.
\end{equation*}
Hence, we see that, modulo denominators of two and additional constraints on poles and exponents of $q$, \Cref{conj: generalised GV} recovers the Gopakumar--Vafa integrality conjecture for products of Calabi--Yau threefolds with the affine plane when $\mgp{T}$ fixes the holomorphic threeform. When $\mgp{T}$ acts non-trivially on the holomorphic threeform, it is expected that for some Calabi--Yau threefolds the \MemInd\ recovers refined Go\-pa\-ku\-mar--Vafa invariants  \cite[Conj.~7.9]{BS24:refined}.

\subsection{Examples}
We apply our vertex formalism to the following geometries with the action by a torus meeting the assumptions of \Cref{thm: vertex formalism}:
\begin{enumerate}
	\item $\Tot{}_X \big(\cL \oplus \cL^\vee\big)$ where $\cL$ is a line bundle on a Calabi--Yau threefold $X$ \hfill (\S\ref{sec: globally sd limit})
	\item $\Tot{}_{\bP^1} \big( \cO(-2) \big)\times \Aaff[3]$\hfill (%\Cref{ex: A1 x C3 1,ex: A1 x C3 2,ex: A1 x C3 3,ex: A1 x C3 4} and
	\S\ref{sec: A1 x C3})
	\item Products of strip geometries with $\Aaff[2]$ \hfill(\S\ref{sec: strips})
	\item $\Tot{}_{\bP^1 \times \bP^1} \big(\cO(-1,0)^{\oplus 2} \oplus \cO(0,-2)\big)$\hfill (\S\ref{sec: RC x A1})
	\item $\Tot{}_{\bP^2} \big(\cO(-1)^{\oplus 3}\big)$ \hfill(\S\ref{sec: O minus 1 plus 3 on P2})
	\item $\Tot{}_{\bP^3} \big(\cO(-2)^{\oplus 2}\big)$ \hfill (\S\ref{sec: O min 2 oplus 2 on P3})
\end{enumerate}
Example (i) is a mild generalisation of the product situation $X\times \Aaff[2]$ discussed before. For (ii) and (iii) we are able to prove closed formulae for the Gromov--Witten series. Based on computer experiments we conjecture a closed form solution for the Gromov--Witten series of (iv) and predict structural properties satisfied by the \MemInds\ of (v) and (vi). See \Cref{conj: closed form RC times A1,conj: O minus 1 plus 3 on P2,conj: O min 2 oplus 2 on P3}.

Out of all geometries studied in this article, example (v) probably best showcases all features of \Cref{conj: generalised GV}. Most notably, for special torus actions the \MemInds\ of this Calabi--Yau fivefold do feature denominators of two. Moreover, in accordance with \Cref{conj: generalised GV} fractions of two seem indeed to be the worst type of denominators to occur. See \Cref{rmk: features of O min 1 oplus 3 on P2} for details.

\subsection{Outline of the paper}
The vertex formalism together with geometric preliminaries is presented in \Cref{sec: vertex formalism}. This section is aimed at a reader who is interested in applying the vertex formalism to a concrete example. We present several such applications in \Cref{sec: examples} while the proof of the vertex formalism is deferred to \Cref{sec: proof vertex formalism}. Our proof follows the idea of Li, Liu, Liu and Zhou \cite{LLLZ09:MathTopVert} of trading the weights in the graph sum resulting from torus localisation for relative Gromov--Witten invariants of partial compactifications of torus orbits in the one-skeleton of the target. This process is referred to as capped localisation in \cite{MOOP11:GWDTtoric}. The proof of \Cref{conj: generalised GV} in the setting where our vertex formalism applies is found in \Cref{sec: generalised GV}. Together with the observation that each weight in the vertex formula is a rational function with monic denominators and numerators having integer coefficients, our proof crucially uses that these features are preserved when taking the plethystic logarithm. This last fact is proven in \Cref{sec: PLog} following ideas of Konishi and Peng \cite{Kon06:Integrality,Pen07:ProofGVConj}.

\subsection{Context \& Prospects}

\subsubsection{Limitations}
Our proof of \Cref{conj: generalised GV} hinges on the fact that we are able to provide closed formulae for every factor in our vertex formula. For this all assumptions made in the statement of \Cref{thm: vertex formalism} are crucial: The torus action being skeletal allows us to use capped localisation. This assumption is therefore responsible for graph sum decomposition of the Gromov--Witten invariant. The assumption that the torus action is \locsd\ reduces vertex and edge terms from five to three dimensions. Together with the fact that the torus action is Calabi--Yau one can identify vertex contributions with the topological vertex weight and evaluate edge terms explicitly. To drop the assumption of being \locsd\ a better understanding of quintuple Hodge integrals and rubber integrals with four Hodge insertions will be required. First steps into this direction will be presented in \cite{GPS26:5Hodge}. See \Cref{rmk: loc sd assumption} for details. Let us also note that the requirement of being \locsd\ prevents us from applying our vertex formalism to interesting examples such as to products $X \times \Aaff[2]$ with a torus action that engineer supersymmetric gauge theories on $\Aaff[2]$ with \textit{general} $\Omega$ background, that is, beyond an anti-diagonal torus action on the affine plane (see \Cref{rmk: gauge th only sd limit}). Finally, the assumption on curve classes to be supported away from \sd\ strata is required to ensure that at most three-legged vertices occur. To drop this assumption, better control over the descendant threefold vertex is required. See \Cref{rmk: support 1,rmk: support 2} for details.

\subsubsection{M2-branes}
\label{sec: denominators of 2}
Let us quickly motivate how the speculation that $\gpT$-equivariant Gromov--Witten invariants of a Calabi--Yau fivefold $Z$ equate to the \smash{$\hat{A}$}-genus of the moduli space of M2-branes of $Z$ implies the statement of \Cref{conj: generalised GV}. In its most optimistic form, the speculation claims the existence of a sufficiently well-behaved moduli space \smash{$\MTwoMod_{\beta}(Z)$} labelled by curve classes in $Z$. The \smash{$\hat{A}$}-genus of this moduli space \smash{$\Omega_\beta = \hat{A}_{\gpT}\big(\MTwoMod_{\beta}(Z)\big)$} should then equate to the Gromov--Witten series via equation \eqref{eq: generalised GV}. Now by Hirze\-bruch--Riemann--Roch, the \smash{$\hat{A}$}-genus lifts to equivariant K-theory. It equates to the Euler characteristic of a square root of the (virtual) canonical bundle
\begin{equation*}
	\Omega_\beta = \chi_{\gpT} \left(\MTwoMod_{\beta}(Z), \cO_{\MTwoMod_{\beta}(Z)}^{\vir} \otimes K^{1/2}_{\vir}\right)\,.
\end{equation*}
In case $K^{1/2}_{\vir}$ is an honest line bundle on the $\gpT$-fixed locus of $\MTwoMod_{\beta}$ the last equality implies that $$\Omega_\beta\in K_{\gpT}(\mr{pt})_{\mr{loc}} \cong \bZ \left[ q_{1}^{\pm 1} , \ldots, q_{\dim \gpT}^{\pm 1} , \left\{(1-{\textstyle\prod_i} q_i^{n_i})^{-1}\right\}_{\boldsymbol{n}\in \bZ^{\dim \gpT}\setminus \{0\}}\right]\,.$$
There are, however, obstructions towards the existence of such a line bundle: First, it may only be well-defined after passing to a cover of $\gpT$ to allow for fractional characters. This is captured in \Cref{conj: generalised GV} by permitting square roots of $\mgp{T}$-characters. Second, the square root of $K_{\vir}$ may only be well-defined as an element in K-theory after inverting two: $\Omega_\beta \in K_{\gpT}(\mr{pt})_{\mr{loc}} \otimes \bZ[\frac{1}{2}]$ (cf.~\cite[Sec.~5.1]{OT23:DT4I}). This explains why we permit denominators of two in \Cref{conj: generalised GV}. Finally, Nekrasov and Okounkov argue that the (virtual) canonical bundle of $\MTwoMod_{\beta}(Z)$ relative to the Chow variety should indeed admit a square root in the Picard group. Now if the $\gpT$-action on $Z$ is skeletal then all fixed points of the Chow variety are isolated. As a consequence, in this setting we should find that $\Omega_\beta$ has integer coefficients. %Note that we precisely verify this prediction in \Cref{thm: generalised GV loc sd} for a special class of geometries.

%We also remark that in the main part of this paper we are going to prove a slightly more refined version of \Cref{conj: generalised GV}: We expect the existence of Nekrasov--Okoun\-kov invariants relative to the Chow variety
%\begin{equation*}
%	\widetilde{\Omega}_\beta \in K_{\gpT}\big( \mr{Chow}_{\beta}(Z) \big)_{\mr{loc}}
%\end{equation*}
%refining equation \eqref{eq: generalised GV}. We verify this expectation in the context of skeletal and \locsd\ torus actions in which the above \MemInds\ are supported on the isolated torus fixed points of the Chow variety. We prove a generalisation of \Cref{conj: generalised GV} for these more refined invariants in \Cref{thm: main part generalised GV loc sd}.

%These are the conceptual motivations for the features of \Cref{conj: generalised GV}. We refer the reader to \Cref{sec: O minus 1 plus 3 on P2} for a concrete example showcasing all these abstract features.\todo{Say that we prove result stronger than \Cref{thm: generalised GV loc sd} over Chow variety.}

\subsubsection{Denominators of two}
Let us expand on the occurrence of denominators of two. In \Cref{thm: vertex formalism} we indeed verify that for skeletal torus actions (satisfying certain additional conditions) \MemInds\ have integer coefficients. In \Cref{sec: O minus 1 plus 3 on P2} we consider the example of a fivefold whose \MemInds\ develop denominators of two precisely when the torus action turns non-skeletal. It appears worthwhile to investigate whether starting from our vertex formula \eqref{eq: vertex formula} one can argue combinatorially that factors of two are indeed the worst type of denominators that may appear under restriction of the torus action.

\subsubsection{Pandharipande--Zinger}
If the moduli space of stable maps to the Calabi--Yau fivefold $Z$ is proper for all genera and curve classes, then the denominators of $\Omega_\beta$ should be even more constrained: For \Cref{conj: generalised GV} to be compatible with the Gopakumar--Vafa integrality conjecture of Pandharipande--Zinger \cite[Conj.~1]{PZ10:CY5} in the non-equivariant limit one must have $\Omega_{\beta} \in \tfrac{1}{8}\bZ$. Whether there is any geometric relation between Pandharipande--Zinger and \MemInds\ is not clear to the author.

\subsubsection{Nekrasov--Okounkov}
Nekrasov and Okounkov conjecture that the generating series of K-theoretic stable pair invariants of a threefold $X$ with an appropriate insertion depending on the choice of two line bundles $\cL_4,\cL_5$ coincides with Laurent expansion of the M2-brane index of $Z=\Tot{}_X \, \cL_4 \oplus \cL_5$ \cite[Conj.~2.1]{NO14:membranes}. Here, the box-counting variable $q$ on the stable pairs side gets identified with the coordinate of $\Gm$ acting anti-diagonally on $\cL_4 \oplus \cL_5$. This conjecturally implies a maps/sheaves correspondence generalising the MNOP conjecture without insertions \cite{MNOP1,MNOP2}. Beyond the situation where the torus action on $X$ is Calabi--Yau, it is not immediately clear, however, how the vertex formalisms governing the respective sides of the correspondence can be related. In order to realise the K-theoretic vertex \cite{NO14:membranes,Arb21:KthDT,KOO21:2legKthVertex} governing the sheaf side in Gromov--Witten theory a better understanding of the stable maps vertex beyond \sd\ torus actions, i.e.~quintuple Hodge integrals, will be crucial. Even the limit of the refined topological vertex \cite{IKV09:RefVert} is currently out of reach in Gromov--Witten theory since it governs a limit where $\epsilon_i \rightarrow \pm \infty$. Accessing this limit would require finding an analytic continuation of the Gromov--Witten series in the torus weights which is currently out of reach. See \Cref{rmk: ref top vert A1 x C3,rmk: ref top vert strip} for comparisons of our vertex formalism with the refined topological vertex in concrete examples.

Conversely, the vertex formalism of \Cref{thm: vertex formalism} is hard to realise in Donaldson--Thomas theory since in general it would require the specialisation of the box-counting variable $q$ to equivariant variables locally at the vertices. This operation is only well-defined after passing to an analytic continuation in $q$ which is currently out of reach beyond the two-leg vertex \cite{KOO21:2legKthVertex}.

\subsubsection{M5-branes}
Suppose now that $Z$ is a toric Calabi-Yau fivefold with a Hamiltonian action by a torus meeting the requirements of \Cref{thm: vertex formalism}. We may factor the affine neighbourhood of a torus fixed point of $Z$ into $\Aaff[3]\times \Aaff[2]$ so that the torus action on both factors is Calabi--Yau. Now suppose that in such an affine neighbourhood we are given a submanifold $L\times \Aaff[1]\times \{0\}$ where $L$ is an Aganagic--Vafa Lagrangian submanifold of $\Aaff[3]$ intersecting a non-compact stratum of $Z$. Then following the arguments of Fang and Liu \cite{FL13:OpenGWtorCY3} one should be able to prove that stable maps from Riemann surfaces with boundary mapping to $L\times \Aaff[1]\times \{0\}$ are enumerated by a vertex formalism similar to \Cref{thm: vertex formalism}. As a consequence the generating series of these open Gromov--Witten invariants should be governed by index type invariants. Such invariants would generalise LMOV invariants \cite{LM01:PolyInvTorKnotTopStr,OV00:KnotInvTopStr,LMV00:KnotsLinksBranes,MV02:FramedKnotsLargeN} and should probably admit an interpretation as indices of M2-M5-bound states. With the methods of \cite{Yu23:BPS} one should be able to prove integrality of these invariants for skeletal and \locsd\ torus actions analogous to \Cref{thm: generalised GV loc sd}.

\subsubsection{Higher dimensions}
By the assumption that each affine neighbourhood of a torus fixed point factors into $\Aaff[3]\times\Aaff[2]$ with the torus acting anti-diagonally on the second factor, the Gromov--Witten vertex reduces from five to three dimensions. Since the threefold vertex admits an explicit formula this observation is the key insight that allows us to prove \Cref{thm: vertex formalism} from which we deduce \Cref{thm: generalised GV loc sd}. One can apply the same trick in arbitrarily high odd dimension. Suppose $Z$ is of dimension $3+2n$ with a skeletal and Calabi--Yau action by a torus $\gpT$ so that locally at every torus fixed point $Z$ looks like $\Aaff[3]\times\Aaff[2] \times \dots \times \Aaff[2]$ with the induced $\gpT$-action on each factor being Calabi--Yau. Under this assumption the generating series of Gromov--Witten invariants of $Z$ can again be evaluated via the topological vertex. As a consequence the generating series admits a presentation as a rational function in variables of the form $\exp (\prod_i \epsilon_i / \prod_j \epsilon'_j)$ where $\epsilon_i$ and $\epsilon'_j$ are torus weights. As in the statement of \Cref{conj: generalised GV} the poles of these rational functions are highly constrained. It is tempting to wonder whether this feature persists for arbitrary torus actions. The author is, however, unaware of any geometry $Z$ showcasing such a feature beyond dimension five.

\subsection*{Acknowledgements}
\begin{sloppypar}
The author benefited from discussions with Alessandro  \mbox{Giacchetto}, Daniel Holmes, Davesh Maulik and Andrei Okounkov. The author was supported by the DFG Walter Benjamin Fellowship 576663726 and the SNF grant SNF-200020-21936.
\end{sloppypar}

%-------------------------------------------------------------------------------

%-------------------------------------------------------------------------------

\section{The vertex formalism}
\label{sec: vertex formalism}
%In this section we will introduce the geometric setup our vertex formalism will apply to and fix our the notation for Gromov--Witten invariants and their generating series.

\subsection{Geometric preliminaries}
Let $Z$ be a smooth quasi-projective Calabi--Yau variety with the action by a torus $\gpT \cong \Gm[m]$. Most of the time we will need to impose the following assumptions on the torus action.

\begin{defn}
	\label{defn: assumptions T action}
	We say that a $\gpT$-action on $Z$ is
	\begin{itemize}
		\item \define{skeletal} if the number of fixed points and one-dimensional orbits is finite;
		
		\item \define{Calabi--Yau} if the action fixes the holomorphic volume form;
		
		\item \define{\locsd} if for each connected component $F$ of the fixed locus of $Z$ there appear at least two non-zero, opposite $\gpT$-weights in the weight decomposition of the normal bundle of $F$ in $Z$.
	\end{itemize}
\end{defn}

For instance, if $Z$ is toric then the action by its dense torus $\widehat{\mgp{T}}$ is skeletal. However, the action by this torus is generally not Calabi--Yau. Only once we pass to a suitable codimension one subtorus \smash{$\mgp{T}\subset \widehat{\mgp{T}}$} the induced action by this torus will be Calabi--Yau. In the context of toric varieties we will often refer to such a torus as the \define{Calabi--Yau torus} of $Z$.

Suppose now we are provided with a skeletal $\gpT$-action on $Z$. We will associate a decorated graph $\Gamma$ to $(Z,\gpT)$ which we call the $\gpT$\define{-diagram} of $Z$. It records the geometry locally around the one-skeleton of $Z$:
\begin{itemize}
	\item vertices $v\in V(\Gamma)$ correspond to $\gpT$-fixed points $p_v$ of $Z$;
	
	\item edges $e \in E(\Gamma)$ correspond to compact $\gpT$-preserved lines $C_e$ each connecting two fixed points;
	
	\item leafs $l\in L(\Gamma)$ correspond to non-compact $\gpT$-preserved lines $C_l$ containing one fixed point;
	
	\item half-edges $h= (v,e) \in H(\Gamma) \cong E(\Gamma)^2 \sqcup L(\Gamma)$ are decorated with the $\gpT$-weight
	\begin{equation*}
		\epsilon_h \coloneqq c_1(T_{p_v} C_e) \in \hhh^2_{\gpT}(\mr{pt},\bZ)
	\end{equation*}
	of the induced torus action on the tangent space of $C_e$ at the fixed point $p_v$.
\end{itemize}

With this notation at hand we remark that a skeletal $\gpT$-action on $Z$ is Calabi--Yau if and only if $0 = \sum_{h\ni v} \epsilon_h$ for any (and thus every) fixed point $v\in V(\Gamma)$. The action is \locsd\ if for every vertex $v$ there exist two half-edges $h\neq h'$ adjacent to $v$ with $\epsilon_h = - \epsilon_{h'}$.

In case the action of the maximal compact subgroup $\gpT_{\bR} \subset \gpT$ on $Z$ is Hamiltonian (with respect to some sufficiently generic symplectic form on $Z$) the moment map
\begin{equation*}
	\mu : Z \longrightarrow \mr{Lie} \gpT_{\bR} \cong \hhh^2_{\gpT}(\mr{pt},\bR) \cong \bR^m
\end{equation*}
provides us with an embedding of $\Gamma$ into $\bR^m$. As in the following example we will often use this embedding for a better visualisation.

\begin{example}
	\label{ex: A1 x C3 1}
	Let us illustrate the setup in a concrete example:
	\begin{equation*}
		Z = \Tot{}_{\bP^1} \big( \cO(-2) \big) \times \Aaff[3]\,.
	\end{equation*}
	Let $\widehat{\gpT} \cong \Gm[5]$ be a dense torus of $Z$. We assume that it acts on the coordinate lines of $\Aaff[3]$ with tangent weights $\epsilon_3,\epsilon_4,\epsilon_5$ respectively. Moreover, we denote by $\epsilon_1$ the tangent weight at $0\in \bP^1 \subset Z$ and finally by $-\epsilon_2$ the $\widehat{\gpT}$-weight on the holomorphic two-form of $\Tot{}_{\bP^1} ( \cO(-2) )$.
	
	The fivefold $Z$ has two fixed points $0,\infty \in \bP^1 \subset Z$ with tangent weights
	\begin{equation*}
		(\epsilon_1, \epsilon_2-\epsilon_1,\epsilon_3,\epsilon_4,\epsilon_5) \qquad \text{and} \qquad (-\epsilon_1, \epsilon_2+\epsilon_1,\epsilon_3,\epsilon_4,\epsilon_5)
	\end{equation*}
	respectively. If we pass to a subtorus $\gpT\cong \Gm[4]$ where the relation $\sum_{i=2}^5 \epsilon_i = 0$ holds then the $\gpT$-action on $Z$ is Calabi--Yau. However, this torus action is not \locsd. To get such an action we have to impose additional constraints for which we have essentially two options: (A) one imposes that $\epsilon_i = -\epsilon_j$ or (B) that $\epsilon_i = \epsilon_1 - \epsilon_2$ and $\epsilon_j = - \epsilon_1 - \epsilon_2$ for some $i,j\in\{3,4,5\}$. We denote the three-dimensional, respectively two-dimensional, subtori realising these constraints by $\gpT_{\mr{A}}$ and $\gpT_{\mr{B}}$. One checks that the action by these tori is indeed \locsd. See \Cref{fig: A1 x C3} for an illustration of the two torus diagrams.
\end{example}

\begin{figure}
	\centering
	\begin{tikzpicture}[scale=1.1]
		\draw (-1,-1) -- (0,0) -- (3,0) -- (4,-1);
		\draw[Cyan] (0,0) -- (0,1.4);
		\draw[Cyan] (3,0) -- (3,1.4);
		\draw[Green] (0,0) -- (-0.5,1.2);
		\draw[Green] (3,0) -- (2.5,1.2);
		\node[Green] at (-0.25,0.6) {$\circ$};
		\node[Green,left] at ($(0,0)!0.75!(-0.5,1.2)$) {$\epsilon_i$};
		\node[Green] at (2.75,0.6) {$\circ$};
		\node[Green,left] at ($(3,0)!0.75!(2.5,1.2)$) {$\epsilon_i$};
		\draw[BurntOrange] (0,0) -- (0.5,-1.2);
		\node[BurntOrange,left] at ($(0,0)!0.75!(0.5,-1.2)$) {$\epsilon_j$};
		\draw[BurntOrange] (3,0) -- (3.5,-1.2);
		\node[BurntOrange,left] at ($(3,0)!0.75!(3.5,-1.2)$) {$\epsilon_j$};
		\node[circle,inner sep=2pt,fill=black] at (0,0){};
		\node[circle,inner sep=2pt,fill=black] at (3,0){};
		\node at (1.5,-2) {(A)}; 
		\draw[white] (7,0) -- (7,1.4);
		\draw (6,-1) -- (7,0) -- (10,0) -- (11,-1);
		\draw[Cyan] (7,0) -- (7,-1.4);
		\draw[Cyan] (10,0) -- (10,-1.4);
		\draw[Green] (7,0) -- (8,1);
		\node[Green] at (7.5,0.5) {$\circ$};
		\node[Green,left] at ($(7,0)!0.75!(8,1)$) {$\epsilon_i$};
		\draw[Green] (10,0) -- (11,1);
		\node[Green,left] at ($(10,0)!0.75!(11,1)$) {$\epsilon_i$};
		\draw[BurntOrange] (7,0) -- (6,1);
		\node[BurntOrange,left] at ($(7,0)!0.75!(6,1)$) {$\epsilon_j$};
		\draw[BurntOrange] (10,0) -- (9,1);
		\node[BurntOrange] at (9.5,0.5) {$\circ$};
		\node[BurntOrange,left] at ($(10,0)!0.75!(9,1)$) {$\epsilon_j$};
		\node[circle,inner sep=2pt,fill=black] at (7,0){};
		\node[circle,inner sep=2pt,fill=black] at (10,0){};
		\node at (8.5,-2) {(B)}; 
	\end{tikzpicture}
	\caption{Illustration of (A) the embedding of the $\gpT_{\mr{A}}$-diagram in $\mr{Lie} \gpT_{\mr{A},\bR} \cong \bR^3$ and (B) the $\gpT_{\mr{B}}$-diagram in $\mr{Lie} \gpT_{\mr{B},\bR} \cong \bR^2$. The blue, green and orange lines indicate the three torus preserved coordinate lines of $\Aaff[3]$. The circle highlights the choice of a \distdir\ at each vertex as will be introduced in \Cref{sec: sign and order}.}
	\label{fig: A1 x C3}
\end{figure}
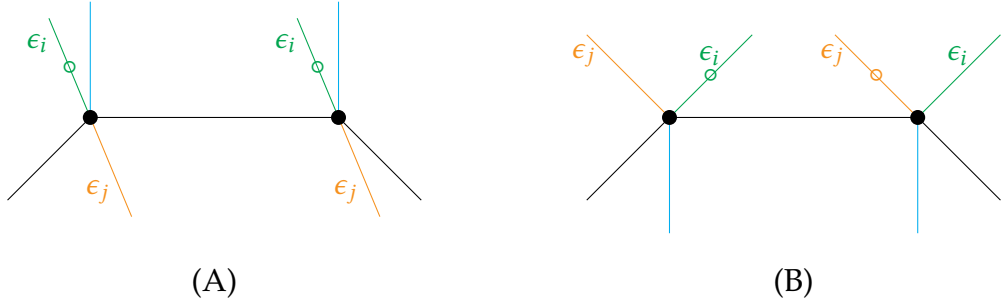

\subsection{Gromov--Witten invariants}
\label{sec: GW invariants}
From now on let $Z$ always denote a Calabi--Yau fivefold together with a skeletal, Calabi--Yau and \locsd\ action by a torus $\gpT$. We consider the $\gpT$-equivariant genus-$g$ Gromov--Witten invariants of $Z$ in curve class $\beta$:
\begin{equation*}
	\GWarg{Z}{\gpT}{g,\beta} \coloneqq \int_{[\Mbar_{g}(Z,\beta)]^{\vir}_{\gpT}} 1 \,.
\end{equation*}
In case the moduli space is not proper this invariant is defined as a $\gpT$-equivariant residue assuming that the $\gpT$-fixed locus is proper.%:
%\begin{equation*}
%	\int_{[\Mbar_{g}(Z,\beta)]^{\vir}_{\gpT}} 1 = \sum_{F_\gamma\subseteq \Mbar_{g}(Z,\beta)^{\gpT}} \int_{[F_\gamma]^{\vir}} \frac{1}{e(N_{F_\gamma,\Mbar_{g}(Z,\beta)}^{\vir})}.
%\end{equation*}
%For a skeletal torus action the connected components of $\Mbar_{g}(Z,\beta)^{\gpT}$ are labelled by certain decorated graphs $\gamma$ whose vertices correspond to components of the domain getting contracted to a fixed point and edges correspond to rational covers of the $\gpT$-preserved lines fully ramified over the fixed points.

The invariants $\GWarg{Z}{\gpT}{g,\beta}$ are rational functions in the torus weights of homogenous degree $2g-2$. Hence, by the degree grading
\begin{equation*}
	\GWarg{Z}{\gpT}{\beta} \coloneqq \sum_{g \geq 0} \GWarg{Z}{\gpT}{g,\beta}
\end{equation*}
is a well-defined formal series. Moreover, we denote
\begin{equation*}
	\GWarg{Z}{\gpT}{} \coloneqq \sum_{\beta \neq 0} Q^\beta ~ \GWarg{Z}{\gpT}{\beta}
\end{equation*}
where the sum runs over all effective curve classes of $Z$ and we write
\begin{equation*}
	1+\sum_{\beta \neq 0} Q^\beta ~ \GWdiscarg{Z}{\gpT}{\beta} \coloneqq \GWdiscarg{Z}{\gpT}{} \coloneqq \exp \GWarg{Z}{\gpT}{}
\end{equation*}
for the generating series of disconnected invariants.

Virtual localisation \cite{GP97:virtloc} decomposes the Gromov--Witten invariants of $Z$ into contributions of the individual $\gpT$-fixed loci of the moduli space:
\begin{equation}
	\label{eq: GW torus localisation}
	\GWdiscarg{Z}{\gpT}{\beta} = \sum_{g \geq 0} \int_{[\Mbar{}_{g}^\bullet(Z,\beta)]^{\vir}_{\gpT}} 1 = \sum_{\gamma} \int_{[F_{\gamma}]^{\vir}} \frac{1}{e_{\gpT}(N^{\vir}_\gamma)}\,.
\end{equation}
The fixed loci are labelled by decorated graphs $\gamma$ whose vertices correspond to components of the domain curve being contracted to a fixed point and edges correspond to rational covers of torus preserved lines fully ramified over the fixed points. Edges are decorated by the degree of their associated cover. We denote by $d_e$ the degree of the covering of the torus orbit in $Z$ labelled by $e\in E(\Gamma)$. This yields an assignment $\boldsymbol{d}:E(\Gamma) \rightarrow \bZ_{\geq 0}$ which is subject to condition $\beta = \sum_e d_e [C_e]$. We call $\boldsymbol{d}$ the \define{skeletal degree} of $\gamma$. We write
\begin{equation*}
	\GWdiscarg{Z}{\gpT}{\boldsymbol{d}} \coloneqq \sum_{\substack{\gamma \text{ has skeletal} \\ \text{degree }\boldsymbol{d}}} ~~ \int_{[F_{\gamma}]^{\vir}} \frac{1}{e_{\gpT}(N^{\vir}_\gamma)} \,.
\end{equation*}
for the partial contribution of the fixed loci with fixed skeletal degree to the overall Gromov--Witten invariant. We use a similar notation to denote partial contributions of connected invariants as well.

\subsection{Combinatorial preparations}
Our vertex formalism will depend on a certain combinatorial choice which we will describe now.

\subsubsection{The choice of \distdir s and orders}
\label{sec: sign and order}
Let $\Gamma$ be the $\gpT$-diagram of $Z$. For each vertex $v\in V(\Gamma)$ we make the following choices. By our assumption that the torus action is \locsd\ there must be (at least) two half-edges $h,h'$ at $v$ whose weights are opposite: $\epsilon_h = - \epsilon_{h'}$. We will call them the \define{\sd\ half-edges} at $v$. Out of the two, we choose a \define{\distdir}  $h_v \in \{h,h'\}$ and write
\begin{equation*}
	\epsilon_v \coloneqq \epsilon_{h_v}\,.
\end{equation*}
Moreover, we fix a \define{cyclic order} $\sigma_v$ of the remaining three half-edges adjacent to $v$. Collectively, we will refer to such a choice $(h_v,\sigma_v)_{v\in V(\Gamma)}$ for every vertex as a choice of \distdir s and orders.

\begin{example}
	\label{ex: A1 x C3 2}
	Let us revisit \Cref{ex: A1 x C3 1} where we considered $Z=\Tot{}_{\bP^1} \big( \cO(-2) \big) \times \Aaff[3]$. We identified two tori $\gpT_{\mr{A}}$ and $\gpT_{\mr{B}}$ whose action on $Z$ is Calabi--Yau and \locsd. The planar embedding of the torus diagrams presented in \Cref{fig: A1 x C3} naturally provides us with a choice of cyclic order after fixing an orientation of $\bR^2$. For the torus $\gpT_{\mr{A}}$ we may choose the same \distdir\ $\epsilon_v = \epsilon_i$ at both fixed points. For $\gpT_{\mr{B}}$ we choose $\epsilon_i$ at $0$ and $\epsilon_j$ at $\infty$. As in \Cref{fig: A1 x C3} we will indicate the choice of a \distdir\ by a circle on the associated stratum.
\end{example}

\subsubsection{Framing and \mxp} % \frmxp
Suppose we fixed a choice of \distdir s and orders. Then to every edge $e=(h,h')$ none of whose half-edges is an \sd\ half-edge we assign what we call a \define{\frp} $f_e\in \bZ/2\bZ$ and a \define{\mxp} $p_e\in \bZ/2\bZ$: Suppose $e$ links the vertices $v$ and $v'$. Then the weights $\epsilon_h$, $\epsilon_{\sigma_v(h)}$, $\epsilon_v$, $\epsilon_{\sigma_{v'}(h')}$ and $\epsilon_{v'}$ either satisfy linear relations
\begin{equation}
	\label{eq: weight reln 1}
	\epsilon_{\sigma_v(h)} = (-1)^{p_1}\epsilon_{\sigma_{v'}(h')} + f_1 \epsilon_h \qquad \text{and} \qquad \epsilon_{v} = (-1)^{p_2}\epsilon_{v'} + f_2 \epsilon_h
\end{equation}
or
\begin{equation}
	\label{eq: weight reln 2}
	\epsilon_{\sigma_v(h)} = (-1)^{p_1}\epsilon_{v'} + f_1 \epsilon_h \qquad \text{and} \qquad \epsilon_{v} = (-1)^{p_2}\epsilon_{\sigma_{v'}(h')} + f_2 \epsilon_h
\end{equation}
for some $f_1,f_2\in \bZ$ and $p_1,p_2\in\bZ/2\bZ$. We define
\begin{equation*}
	f_e \equiv f_1+ f_2  \,,\qquad p_e \equiv p_1 + p_2\,.
\end{equation*}

There is an alternative characterisation of the \mxp\ which is often more practical: The normal bundle of the line $C_e$ splits as
\begin{equation*}
	N_{C_e} Z \cong \cO_{\bP^1}(a_1) \oplus \cO_{\bP^1}(a_2) \oplus \cO_{\bP^1}(a_3) \oplus \cO_{\bP^1}(a_4)\,.
\end{equation*}
Every factor can be associated with one of the half-edges other than $h$ and $h'$ at each of the fixed points $v$ and $v'$. Hence, the splitting induces a bijection
\begin{equation*}
	\alpha: \big\{h_v,h_{\sigma(h)},h_v',h_{\sigma^2(h)}\big\} \isomrightarrow \big\{1,2,3,4\big\} \isomrightarrow \big\{h_{v'},h_{\sigma(h')},h_{v'}',h_{\sigma^2(h')}\big\}
\end{equation*}
between the half-edges at $v$ and $v'$. Then the \mxp\ of $e$ can be identified with the cardinality
\begin{equation*}
	p_e \equiv \big|\alpha\big(\{h_v,h_{\sigma(h)}\}\big) \cap \{h_{v'},h_{\sigma(h')}\}\big| \,.
\end{equation*}

\begin{example}
	\label{ex: A1 x C3 4}
	We continue our running example $Z=\Tot{}_{\bP^1} \big( \cO(-2) \big) \times \Aaff[3]$. There is one compact edge $e$ associated with the zero section $\bP^1 \subset Z$. One checks that in both cases (A) and (B) $f_e$ is even while $p_e$ is odd.
\end{example}

\subsection{Diagrammatic rules}
In this section we will state the diagrammatic rules that allow the evaluation of the Gromov--Witten invariants  $\GWdiscarg{Z}{\gpT}{\boldsymbol{d}}$ after having fixed a choice of \distdir s and orders. There is, however, a technical condition we need to impose on the skeletal degree $\boldsymbol{d}$.

\begin{defn}
	\label{defn: curve class condition}
	We say that $\boldsymbol{d}\in\bZ_{\geq 0}^{E(\Gamma)}$ is \define{supported away from \sd\ strata} if $d_e=0$ whenever at least one of the half-edges $h$, $h'$ of $e=(h,h')$ is an \sd\ half-edge.
\end{defn}

\begin{example}
	\label{ex: A1 x C3 3}
	With the \distdir\ and order we fixed in \Cref{ex: A1 x C3 2} for our running example we can infer from \Cref{fig: A1 x C3} that in both cases (A) and (B) any multiple of the zero section is supported away from \sd\ strata as all \sd\ half-edges are non-compact legs.
\end{example}

\subsubsection{Partition labels}
We are now ready to state the rules of our vertex formalism. We decorate each half-edge $h$ of the $\gpT$-diagram $\Gamma$ with a partition $\mu_h$ subject to the following conditions:
\begin{itemize}
	\item $\mu_h = \varnothing$ whenever $h$ is a leaf or an \sd\ half-edge;
	
	\item for each compact edge $e=(h,h')\in E(\Gamma)$ we have $\mu_h = \mu_{h'}$ if the \mxp\ is even and $\mu_h = \mu_{h'}^{\transpose}$ otherwise;
	
	\item for all edges $e=(h,h')$ we have $|\mu_h| = |\mu_{h'}| = d_e$.
\end{itemize}
We denote the set of all decorations of half-edges by partitions satisfying the above conditions by $\cP_{\Gamma, \boldsymbol{p}, \boldsymbol{d}}$. Note that through the \mxp , the second condition depends on the choice of \distdir s and orders we fixed.

\subsubsection{Edge weights}
Given a partition label $\boldsymbol{\mu} \in \cP_{\Gamma, \boldsymbol{p}, \boldsymbol{d}}$ we assign a weight to each edge $e$ as follows: Denote the half-edges associated to this edge by $h=(v,e)$ and $h'=(v',e)$. Then $e$ gets assigned the weight
\begin{equation*}
	E(e,\boldsymbol{\mu}) \coloneqq (-1)^{(f_e+p_e)|\mu_h|} ~ \exp \left(\frac{\kappa({\mu_h})}{2} \, \frac{\epsilon_v \epsilon_{\sigma(h)} + (-1)^{p_e+1} \epsilon_{v'} \epsilon_{\sigma(h')}}{\epsilon_h}\right)
\end{equation*}
where $\kappa(\mu) \coloneqq \sum_{i=1}^{\ell(\mu)} \mu_i (\mu_i - 2i + 1)$ denotes the second Casimir invariant.

\subsubsection{Vertex weights}
Given three partitions $\mu_1$, $\mu_2$ and $\mu_3$ we introduce the topological vertex function
\begin{equation}
	\label{eq: TV formula}
	\cW_{\mu_1,\mu_2,\mu_3}(q) = q^{\kappa(\mu_1)/2} s_{\mu_3\,}\!\big(q^\rho\big) \sum_{\nu} s_{\frac{\mu_1^{\transpose}}{\nu}} \!\big(q^{\rho+\mu_3}\big) \, s_{\frac{\mu_2}{\nu}\,} \!\big(q^{\rho+\mu_3^{\transpose}}\big)\,.
\end{equation}
Here, $s_{\alpha/\beta}(q^{\rho+\gamma})$ denotes the skew Schur function $$s_{\frac{\alpha}{\beta}}(x_1,x_2,\ldots)$$ evaluated at $x_i = q^{-i+1/2+\gamma_i}$. A priori, this makes \smash{$s_{\alpha/\beta}(q^{\rho+\gamma})$} a formal Laurent series in \smash{$q^{-1/2}$}. It can, however, be shown that the series converges to a rational function implying that also \eqref{eq: TV formula} is a rational function in \smash{$q^{1/2}$}. (We will recall this fact in more detail in the proof of \Cref{thm: main part generalised GV loc sd}.)

Now given a partition label $\boldsymbol{\mu} \in \cP_{\Gamma, \boldsymbol{p}, \boldsymbol{d}}$ we assign a weight to each vertex $v$ as follows: Remember that we fixed a cyclic permutation $\sigma_v = (h_1 \, h_2 \, h_3)$ of three half-edges at $v$ and that our choice of \distdir\ singled out a distinct torus weight $\epsilon_v$. The weight we assign to $v$ is
\begin{equation*}
	\cW(v,\boldsymbol{\mu}) \coloneqq \cW_{ \mu_{h_1},\mu_{h_2},\mu_{h_3} }(\re^{\epsilon_v})\,.
\end{equation*}
This assignment is well defined since the topological vertex function is invariant under cyclic permutations of partitions.

\subsection{The vertex formula}
With these diagrammatic rules at our disposal we are finally able to state the first main result of this paper.
\begin{thm}
	\label{thm: vertex formalism main part}
	Let $Z$ be a Calabi--Yau fivefold with a skeletal, Calabi--Yau and \locsd\ action by a torus $\gpT$. Fix a choice of \distdir s and orders. Then for all skeletal degrees $\boldsymbol{d}$ supported away from \sd\ strata we have
	\begin{equation}
		\label{eq: vertex formula main part}
		\GWdiscarg{Z}{\gpT}{\boldsymbol{d}} = \sum_{\boldsymbol{\mu}\in \cP_{\Gamma, \boldsymbol{p}, \boldsymbol{d}} } ~ \prod_{e \in E(\Gamma)} E(e,\boldsymbol{\mu}) \prod_{v \in V(\Gamma)} \cW(v,\boldsymbol{\mu})\,.
	\end{equation}
\end{thm}

We recover the disconnected Gromov--Witten invariants in a curve class $\beta$ by summing over all skeletal degrees $\boldsymbol{d}$ satisfying $\beta = \sum_e d_e [C_e]$:
\begin{equation*}
	\GWdiscarg{Z}{\gpT}{\beta} = \sum_{\boldsymbol{d}} \GWdiscarg{Z}{\gpT}{\boldsymbol{d}}\,.
\end{equation*}
To be able to apply \Cref{thm: vertex formalism main part} one has to assume that none of the skeletal degrees in the above sum is supported on \sd\ strata. In this case we say that $\beta$ is supported away from \sd\ strata.

\begin{cor}
	\label{cor: vertex formalism curve class}
	(\Cref{thm: vertex formalism}) Let $Z$ be a Calabi--Yau fivefold with a skeletal, Calabi--Yau and \locsd\ action by a torus $\gpT$. Fix a choice of \distdir s and orders. Then for all effective curve classes $\beta$ supported away from \sd\ strata we have
	\begin{flalign}
		\label{eq: vertex formula main part curve class}
		&&\GWdiscarg{Z}{\gpT}{\beta} = \sum_{\boldsymbol{d}} \sum_{\boldsymbol{\mu}\in \cP_{\Gamma, \boldsymbol{p}, \boldsymbol{d}} } ~ \prod_{e \in E(\Gamma)} E(e,\boldsymbol{\mu}) \prod_{v \in V(\Gamma)} \cW(v,\boldsymbol{\mu})\,. && \qed
	\end{flalign}
\end{cor}

\begin{rmk}
	\label{rmk: weaker condition curve class}
	We remark that it actually suffices to impose a slightly weaker condition on the curve class $\beta$: Suppose that for every $\boldsymbol{d}$ satisfying $\beta = \sum_e d_e [C_e]$ we have $\GWdiscarg{Z}{\gpT}{\boldsymbol{d}}=0$ whenever there is an \sd\ edge $e$ with $d_e>0$. Then in this case \eqref{eq: vertex formula main part curve class} still holds true with the first sum only ranging over those skeletal degrees that are supported away from \sd\ strata. In \Cref{sec: O min 2 oplus 2 on P3} we will see an application where this extra freedom is indeed crucial.
\end{rmk}

The proof of \Cref{thm: vertex formalism main part} is deferred to \Cref{sec: proof vertex formalism}. In the remaining parts of this section we will first explain how the vertex formula implies \Cref{conj: generalised GV} and second we will compare our formula with the original vertex formalism for toric Calabi--Yau threefolds. To see the vertex formula at work in several examples we refer the reader to \Cref{sec: examples}.

\subsection{\texorpdfstring{On \Cref{conj: generalised GV}}{On Conjecture A}}
\label{sec: generalised GV}
We will prove \Cref{conj: generalised GV} in the setting where our vertex formalism applies in a slightly stronger version than it was stated in \Cref{thm: generalised GV loc sd} in the introduction.

\begin{thm}
	\label{thm: main part generalised GV loc sd}
	Let $Z$ be a Calabi--Yau fivefold with a skeletal, Calabi--Yau and \locsd\ action by a torus $\gpT$. Fix a basis $H^2_{\mgp{T}}(\mr{pt},\bZ)\cong \bZ[\epsilon_1,\ldots,\epsilon_{m}]$ and a choice of \distdir s and orders. Then there exist rational functions
	\begin{equation*}
		\Omega_{\boldsymbol{d}}(q_i) \in \bZ \left[ q_{1}^{\pm 1/2} , \ldots,  q_{m}^{\pm 1/2} , \, \left\{\big(1-{\textstyle\prod_i} q_i^{n_i/2} \big)^{-1}\right\}_{\boldsymbol{n}\in \bZ^{m}\setminus \{0\}} \,\right]
	\end{equation*}
	labelled by skeletal degrees $\boldsymbol{d}$ supported away from \sd\ strata such that under the change of variables $q_i= \re^{\epsilon_i}$ we have
	\begin{equation*}
		\GWarg{Z}{\gpT}{\boldsymbol{d}} = \sum_{%\substack{k\in\bZ_{> 0}\\
				k \vertspace \boldsymbol{d}%}
		} \frac{1}{k} \, \Omega_{\boldsymbol{d}/k}\big(q_i^k\big)\,.
	\end{equation*}
\end{thm}

\begin{proof}
	We may identify $R=\bZ [ q_{1}^{\pm 1/2} , \ldots, q_{m}^{\pm 1/2} , \, \{(1-{\textstyle\prod_i} q_i^{n_i/2} )^{-1}\}_{\boldsymbol{n}} ]$ with the ring of virtual representations of a double cover of $\mgp{T}$ localised at the augmentation ideal. Since by \Cref{lem: PLog preserves integrality} the plethystic logarithm maps a series with coefficients in $R$ to one with coefficients in $R$, it suffices to prove that
	\begin{equation*}
		\GWdiscarg{Z}{\gpT}{\boldsymbol{d}} \in R
	\end{equation*}
	for all $\boldsymbol{d}$ supported away from \sd\ strata. This is true if we show that actually every individual factor in our vertex formula \eqref{eq: vertex formula main part} is an element in $R$. Indeed, for edge terms this is a consequence of the fact that $\epsilon_h$ divides $\epsilon_v \epsilon_{\sigma(h)} - (-1)^{p_e} \epsilon_{v'} \epsilon_{\sigma(h')}$ for all edges $e=(h,h')$ by the linear relations \eqref{eq: weight reln 1} and \eqref{eq: weight reln 2}. This can be seen in a case-by-case analysis. As a consequence, we can write an edge term as
	\begin{equation*}
		E(e,\boldsymbol{\mu}) = (-1)^{(f_e+p_e)|\mu_h|} q_h^{\kappa(\mu_h)/2}
	\end{equation*}
	where $q_h$ is a $\gpT$-character. Since $\kappa(\mu_h)$ is integer, we thus deduce that every edge weight lies in $R$.
	
	Regarding vertex weights, let us show that $\cW_{\mu_1,\mu_2,\mu_3}(q)$ is a rational function in $q^{1/2}$ with poles only at zero and roots of unity. Indeed, up to a leading monomial factor, the topological vertex depends on $q$ only through the specialised skew Schur functions \smash{$s_{\alpha/\beta}(q^{\rho+\mu})$}. The latter are uniquely determined from the specialisation of the power functions $p_k$. The rationality and pole constraint claimed thus follows from the evaluation
	\begin{equation*}
		p_k\big(q^{\rho+\mu}\big) = \frac{1}{q^{k/2}-q^{-k/2}} + \sum_{i=1}^{\ell(\mu)} q^{(-i+1/2+\mu_i)k} - q^{(-i+1/2)k}\,.
	\end{equation*}
	Alternatively, one may also interpret the rationality and the restriction of poles in the vertex weight as a consequence of the relative Gromov--Witten/Donaldson--Thomas correspondence for toric threefolds \cite[Thm.~1 \& 3]{MOOP11:GWDTtoric}.
\end{proof}

Again, we obtain the analogue of \Cref{thm: main part generalised GV loc sd} for invariants labelled by curve classes by summing over skeletal degrees.

\begin{cor}  (\Cref{thm: generalised GV loc sd})
	\label{thm: main part generalised GV curve class loc sd}
	Under the assumptions of \Cref{thm: main part generalised GV loc sd} there exist \mbox{rational} functions
	\begin{equation*}
		\Omega_{\beta}(q_i) \in \bZ \left[  q_{1}^{\pm 1/2} , \ldots,  q_{m}^{\pm 1/2} , \, \left\{\big(1-{\textstyle\prod_i} q_i^{n_i/2} \big)^{-1}\right\}_{\boldsymbol{n}\in \bZ^{m}\setminus \{0\}} \,\right]
	\end{equation*}
	labelled by curve classes $\beta$ supported away from \sd\ strata such that under the change of variables $q_i= \re^{\epsilon_i}$ we have
	\begin{flalign*}
		&& \GWarg{Z}{\gpT}{\beta} = \sum_{%\substack{k\in\bZ_{> 0}\\
				k \vertspace \beta%}
		} \frac{1}{k} \, \Omega_{\beta/k}\big(q_i^k\big)\,. && \qed
	\end{flalign*}
\end{cor}

\begin{rmk}
	The same result holds under the slightly weaker but more technical assumption stated in \Cref{rmk: weaker condition curve class}.
\end{rmk}

\begin{rmk}
	\label{rmk: refined generalised GV}
	The fact that a lift of the Gromov--Witten series labelled by skeletal degrees to a rational function exists by \Cref{thm: main part generalised GV loc sd} suggests there should be a refinement of \Cref{conj: generalised GV} along the following lines. The yet-to-be-constructed moduli space $\MTwoMod_{\beta}(Z)$ of M2-branes in curve class $\beta$ should admit a morphism to the Chow variety. Also the moduli space of stable maps admits such a morphism by taking the support of a stable map:
	\begin{equation*}
	\begin{tikzcd}
		\MTwoMod_{\beta}(Z) \ar[dr] & & \Mbar{}^{\bullet}(Z,\beta) \ar[dl]\\
		&\mr{Chow}_{\beta}(Z)&
	\end{tikzcd}
	\end{equation*}
	Pushing forward an appropriate K-theory class along the left arrow should yield an element		
	\begin{equation*}
		\widehat{\Omega}_\beta \in K_{\gpT}\big( \mr{Chow}_{\beta}(Z) \big)_{\mr{loc}}\,.
	\end{equation*}
	It should be related to the push-forward of the virtual fundamental class along the right arrow via the Chern character and the plethysm on the Chow variety (cf.~\cite[Sec.~2.3.5]{NO14:membranes}) so that \Cref{conj: generalised GV} is recovered by pushing the refined identity forward to a point. Since for skeletal torus actions all fixed points of $\mr{Chow}_{\beta}(Z)$ are isolated and labelled by the skeletal degrees, we see that with \Cref{thm: main part generalised GV loc sd} we actually proved such a refined version of \Cref{conj: generalised GV}.
\end{rmk}

\subsection{The globally \sd\ situation}
\label{sec: globally sd limit}
In this section we will specialise our vertex formalism to compute the local contribution of a Calabi--Yau threefold $X\subset Z$ which is the fixed locus of a Calabi--Yau $\Gm_q$-action on $Z$.

\subsubsection{The general case}
Let us describe the local setup. Let $X$ be a smooth toric Calabi--Yau threefold together with the action by its Calabi--Yau torus $\gpT'\cong \Gm[2]$. Let $\cL$ be a $\gpT'$-equivariant line bundle on $X$ and let $\Gm_q$ act on its fibres with character $q^{-1}$. We obtain a Calabi--Yau torus action on the local Calabi--Yau fivefold
\begin{equation*}
	Z \coloneqq \Tot{}_X \cL \oplus \cL^\vee \qquad \reflectbox{\actson} \qquad \gpT' \times \Gm_q \eqqcolon \gpT\,.
\end{equation*}

To apply our vertex formalism in this situation, we first fix a specific choice of \distdir s and orders: For each vertex $v$ of the $\gpT$-diagram, that is for each torus fixed point $p_v\in X$, we choose the \distdir\ to be the half-edge associated with the line bundle $\cL$:
\begin{equation*}
	\epsilon_v = - c_1(\cL|_{p_v})\,.
\end{equation*}
Assuming that the induced action by $\gpT'_{\bR}$ is Hamiltonian, the moment map $\mu: X \rightarrow \hhh^2_{\gpT'}(\mr{pt},\bR)$ yields an embedding of the $\gpT$-diagram into $\hhh^2_{\gpT'}(\mr{pt},\bR) \cong \bR^2$. Hence, fixing an orientation of $\bR^2$ yields a cyclic order for the three half-edges associated to $T_{p_v}X$ at each vertex $v$. With this choice of \distdir s and orders the \mxp\ of each edge is odd.
To determine the edge weights, note that the normal bundle of each torus preserved line splits into
\begin{equation*}
	N_{C_e}X \oplus \cL|_{C_e}\oplus \cL^\vee|_{C_e} \cong \cO_{\bP^1}(-1+f_{e,1}) \oplus \cO_{\bP^1}(-1-f_{e,1}) \oplus \cO_{\bP^1}(f_{e,2}) \oplus \cO_{\bP^1}(-f_{e,2})
\end{equation*}
for some $f_{e,1},f_{e,2} \in \bZ$. Writing $h=(v,e)$ and $h'=(v',e)$ for the half-edges of $e$ the torus weights at the fixed points are related by
\begin{equation*}
	\epsilon_{\sigma_v(h)} = - \epsilon_{\sigma_{v'}(h')} + f_{e,1} \epsilon_h\,, \qquad \epsilon_v = \epsilon_{v'} + f_{e,2} \epsilon_h\,.
\end{equation*}
Thus, writing $q_i = \re^{\epsilon_i}$ our vertex formula \eqref{eq: vertex formula main part} specialises to
\begin{equation}
\label{eq: formula globally sd}
\begin{split}
	& \GWdiscarg{Z}{\mgp{T}}{\boldsymbol{d}} = \\
	& \sum_{\boldsymbol{\mu}\in \cP_{\Gamma, \boldsymbol{p}, \boldsymbol{d}}} \prod_{e \in E(\Gamma)} (-1)^{(f_{e,1}+f_{e,2}+1)|\mu_h|} \, \left(q_v^{f_{e,1}} q_{\sigma_v(h)}^{f_{e,2}} q_h^{ - f_{e,1}f_{e,2}}\right)^{\kappa({\mu_h}) / 2}  \prod_{v \in V(\Gamma)} \cW_{\mu_{h^v_{1}},\mu_{h^v_{2}},\mu_{h^v_{3}}}(q_v) \,.
\end{split}
\end{equation}

\subsubsection{The threefold limit}
\label{sec: TV limit}
Let us further specialise to the case where $\cL \cong \cO_X$ or in other words to the situation where $Z$ is the product of $X$ with the affine plane:
\begin{equation*}
	Z=X\times \Aaff[2]\,.
\end{equation*}
The torus $\Gm_q$ acts anti-diagonally on the affine plane with weights $\pm \epsilon \coloneqq \pm c_1(q)$. As explained in \cite[Sec.~2.4]{BS24:refined}, in this situation the $\gpT$-equivariant Gromov--Witten invariants of $Z$ recover the ones of the threefold $X$ where the weight $\epsilon$ takes the role of the genus counting variable:
\begin{equation}
	\label{eq: GW CY5 to CY3}
	\GWdiscarg{Z}{\mgp{T}}{\boldsymbol{d}} = \sum_{g \in \bZ} (-\epsilon^2)^{g-1} ~ \GWdiscarg{X}{\mgp{T}}{\boldsymbol{d}}\,.
\end{equation}
Specialising $\epsilon_v=\epsilon$ and $f_{e,2}=0$ in equation \eqref{eq: formula globally sd}, our vertex formalism thus yields the following formula for these Gromov--Witten invariants:
\begin{equation*}
	\sum_{g \in \bZ} (-\epsilon^2)^{g-1} ~ \GWdiscarg{X}{\mgp{T}}{\boldsymbol{d}} = \sum_{\boldsymbol{\mu}\in \cP_{\Gamma, \boldsymbol{p}, \boldsymbol{d}}} ~ \prod_{e \in E(\Gamma)} (-1)^{(f_e+1)|\mu_h|} q^{\kappa({\mu_h})  f_e /2}   \prod_{v \in V(\Gamma)} \cW_{\mu_{h^v_{1}},\mu_{h^v_{2}},\mu_{h^v_{3}}}(q) \,.
\end{equation*}
This is precisely the topological vertex formula for toric Calabi--Yau threefolds of Aganagic--Klemm--Mari\~{n}o--Vafa \cite{AKMV05:TopVert} as stated in \cite{LLLZ09:MathTopVert}. Note also that in the product case $X\times \Aaff[2]$ one may identify $f_e$ with what is usually called the framing factor; this is why in the general case we chose to call its congruence class modulo two the \frp .

%-------------------------------------------------------------------------------
\section{Examples}
\label{sec: examples}

\subsection{\texorpdfstring{$\boldsymbol{\Tot{}_{\bP^1} ( \cO(-2) ) \times \Aaff[3]}$}{Tot O_P1(-2) x C3} (continued)}
\label{sec: A1 x C3}
Let us apply the vertex formalism to our running example $Z=\Tot{}_{\bP^1} ( \cO(-2) ) \times \Aaff[3]$ (\Cref{ex: A1 x C3 1,ex: A1 x C3 2,ex: A1 x C3 3,ex: A1 x C3 4}). First, we consider case (A) which is a special instance of the situation described in \Cref{sec: TV limit}: We have $Z = X\times \Aaff[2]$ and $\gpT_{\mr{A}}$ acts on the coordinate lines of the affine plane with opposite weights. Hence, the formalism reduces to the usual topological vertex method which yields
\begin{equation}
	\label{eq: A1 x C3 formula A}
	\GWdiscarg{Z}{\gpT_{\mr{A}}}{} = \sum_{\mu} Q^{|\mu|} q_i^{\kappa(\mu) /2} s_{\mu}\big(q_i^{\rho}\big) \, s_{\mu^{\transpose}}\big(q_i^{\rho}\big) = \Exp \frac{Q}{\big(q_i^{1/2} - q_i^{-1/2}\big)^2}
\end{equation}
where we write $q_i = \re^{\epsilon_i}$ and $\Exp$ denotes the plethystic exponential
\begin{equation*}
	\Exp \big( f(q,Q)\big) \coloneqq \exp \left( \sum_{k>0} \frac{1}{k} \, f\big(q^k,Q^k\big) \right)\,.
\end{equation*}

Situation (B) is more interesting. Here, the \mxp\ of the unique compact edge is odd too but now the vertices carry different choices for the \distdir . If we apply \Cref{thm: vertex formalism main part} we get the formula
\begin{equation}
	\GWdiscarg{Z}{\gpT_{\mr{B}}}{} = \sum_{\mu} Q^{|\mu|} q_i^{-\kappa(\mu)/2} \, s_{\mu}\big(q_i^{-\rho}\big) \, s_{\mu^{\transpose}}\big(q_j^{\rho}\big) = \Exp \frac{Q}{\big(q_i^{1/2} - q_i^{-1/2}\big)\big(q_j^{1/2} - q_j^{-1/2}\big)}\,.
\end{equation}
Formula \eqref{eq: A1 x C3 formula A} and the above are both specialisations of the following conjectural formula for the four dimensional Calabi--Yau torus $\gpT$ acting on $Z$ \cite[Conj.~3.4]{BS24:refined}:
\begin{equation}
	\label{eq: A1 x C3 conj formula}
	\GWdiscarg{Z}{\gpT}{} = \Exp \frac{\big(q_2^{1/2} - q_2^{-1/2}\big) ~ Q}{\big(q_3^{1/2} - q_3^{-1/2}\big)\big(q_4^{1/2} - q_4^{-1/2}\big)\big(q_5^{1/2} - q_5^{-1/2}\big)}\,.
\end{equation}

\begin{rmk}
	\label{rmk: ref top vert A1 x C3}
	As already explained in \cite[Sec.~7.2.5]{BS24:refined}, the fact that neither of the formulae for torus actions (A) and (B) agree with any quantity computed via the refined topological vertex \cite{IKV09:RefVert} is due to the non-compactness of the moduli space of stable maps to the threefold $\Tot{}_{\bP^1} (\cO(-2)) \times \Aaff$. The refined topological vertex evaluates the $\epsilon_3 \rightarrow \pm \infty$ limit of formula \eqref{eq: A1 x C3 conj formula}. In the following section we will analyse a larger class of toric threefolds exhibiting the same phenomenon.
\end{rmk}

\medskip

\subsection{Strip geometries}
\label{sec: strips}
Generalising the last example, our formalism applies to a wider class of so-called strip geometries. By this we mean a product $Z=X\times\Aaff[2]$ where $X$ is the toric variety whose fan is the cone over a triangulated strip
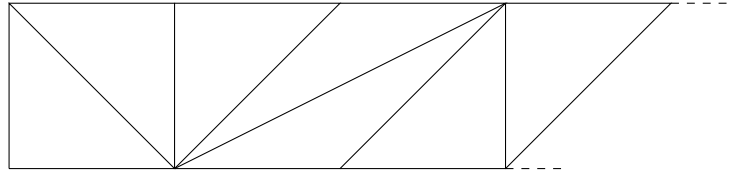
\begin{equation*}
	\begin{tikzpicture}[scale=1.1]
		\draw (0,0) -- (2,0) -- (0,2) -- (0,0);
		\draw (2,0) -- (2,2) -- (0,2);
		\draw (2,0) -- (4,2) -- (2,2);
		\draw (2,0) -- (6,2) -- (4,2);
		\draw (2,0) -- (4,0) -- (6,2);
		\draw (4,0) -- (6,0) -- (6,2);
		\draw (6,0) -- (8,2) -- (6,2);
%		\draw (6,0) -- (8,0) -- (8,2);
		\draw[dashed] (6,0) -- (6.7,0);
		\draw[dashed] (8,2) -- (8.7,2);
	\end{tikzpicture}
\end{equation*}
placed at height one. The torus diagram of $X$ takes the following shape:
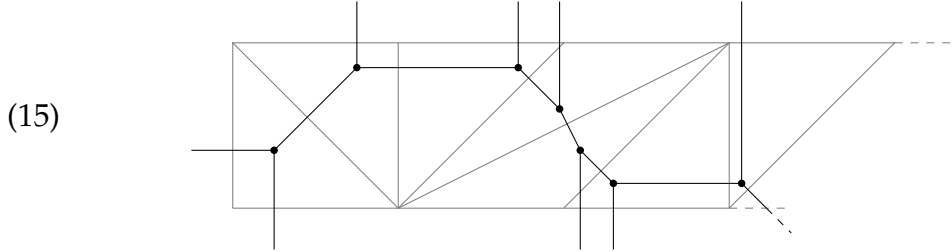
\begin{equation}
	\label{eq: strip toric diagram}
	\begin{tikzpicture}[baseline={(current bounding box.center)}, scale=1.1]
		\draw[gray] (0,0) -- (2,0) -- (0,2) -- (0,0);
		\draw[gray] (2,0) -- (2,2) -- (0,2);
		\draw[gray] (2,0) -- (4,2) -- (2,2);
		\draw[gray] (2,0) -- (6,2) -- (4,2);
		\draw[gray] (2,0) -- (4,0) -- (6,2);
		\draw[gray] (4,0) -- (6,0) -- (6,2);
		\draw[gray] (6,0) -- (8,2) -- (6,2);
%		\draw[gray] (6,0) -- (8,0) -- (8,2);
		\draw[gray, dashed] (6,0) -- (6.7,0);
		\draw[gray, dashed] (8,2) -- (8.7,2);
		\draw (-0.5,0.7) -- (0.5,0.7) -- (1.5,1.7) -- (3.45,1.7) -- (3.95,1.2) -- (4.2,0.7) -- (4.6,0.3) -- (6.15,0.3) -- (6.45,0);% -- (8,0.15);
		\draw[dashed] (6.45,0) -- (6.75,-0.3);
		\draw (0.5,0.7) -- (0.5,-0.5);
		\draw (1.5,1.7) -- (1.5,2.5);
		\draw (3.45,1.7) -- (3.45,2.5);
		\draw (3.95,1.2) -- (3.95,2.5);
		\draw (4.2,0.7) -- (4.2,-0.5);
		\draw (4.6,0.3) -- (4.6,-0.5);
		\draw (6.15,0.3) -- (6.15,2.5);
%		\draw (6.3,0.15) -- (6.3,-0.5);
		%
		\node[circle,inner sep=1pt,fill=black] at (0.5,0.7){};
		\node[circle,inner sep=1pt,fill=black] at (1.5,1.7){};
		\node[circle,inner sep=1pt,fill=black] at (3.45,1.7){};
		\node[circle,inner sep=1pt,fill=black] at (3.95,1.2){};
		\node[circle,inner sep=1pt,fill=black] at (4.2,0.7){};
		\node[circle,inner sep=1pt,fill=black] at (4.6,0.3){};
		\node[circle,inner sep=1pt,fill=black] at (6.15,0.3){};
%		\node[circle,inner sep=1pt,fill=black] at (6.3,0.15){};
	\end{tikzpicture}
\end{equation}
Here, we presented its embedding in $\bR^2$ provided by the moment map of the two-dimensional Calabi--Yau torus of $X$. Now let $\mgp{T}'\cong \Gm[4]$ be a Calabi--Yau torus acting on $Z$. Note that by construction all $\mgp{T}'$-weights attached to upwards pointing legs coincide. We denote their weight by $\epsilon_2$. Similarly, let us write $\epsilon_3$ for the $\mgp{T}'$-weight associated to downwards pointing legs. Finally, denote the tangent weights at the origin of $\Aaff[2]$ by $\epsilon_4$ and $\epsilon_5$. We remark that in $\mgp{T}'$-equivariant cohomology $\epsilon_2,\ldots,\epsilon_5$ are linearly independent.

Now consider the two-dimensional subtorus $\gpT\subset \gpT'$ on which the relations $\epsilon_2 = -\epsilon_4$ and $\epsilon_3 = -\epsilon_5$ hold. With these constraints the induced $\gpT$-action on $Z$ is \locsd\ and we can apply \Cref{thm: vertex formalism main part}. For this we fix a \distdir\ at each vertex by choosing $\epsilon_v = \epsilon_4$ whenever a vertex $v$ carries an upwards pointing leg and $\epsilon_v = -\epsilon_5$ otherwise. A choice of \distdir\ at each vertex is fixed by choosing an orientation of the plane $\bR^2$ for the embedding of the diagram \eqref{eq: strip toric diagram}. With this choice the \mxp\ of an edge is even if and only if it connects two vertices where one has an upwards and the other one a downwards pointing leg attached.

To state the vertex formula we label the vertices in \eqref{eq: strip toric diagram} from left to right by $v_1,\ldots,v_N$. The index set decomposes into $I_{\mr{u}} \sqcup I_{\mr{d}} = \{1,\ldots,N\}$ labelling vertices with an upwards respectively downwards pointing leg. We decorate the edge $e$ connecting $v_i$ with $v_{i+1}$ with a partition $\mu_i$ and write $Q_i = Q^{C_e}$. With this notation \Cref{thm: vertex formalism main part} yields the formula
\begin{equation*}
	\GWdiscarg{Z}{\gpT}{} = \sum_{\mu_1,\ldots,\mu_{N+1}} \,  \prod_{i=1}^{N-1}Q_i^{|\mu_i|} q_{i,i+1}^{-\kappa(\mu_i)/2} ~ \prod_{i\in I_{\mr{u}}} \cW_{\mu_{i},\mu_{i-1}^{\transpose},\varnothing}\big(q_4\big) ~ \prod_{i\in I_{\mr{d}}} \cW_{\mu_{i-1}^{\transpose},\mu_{i},\varnothing}\big(q_5^{-1}\big) 
\end{equation*}
where
\begin{equation*}
	q_{i,i+1} = \begin{cases}
		q_4 & i,i+1\in I_{\mr{u}} \\
		q_5 & i,i+1\in I_{\mr{d}} \\
		q_4 q_5 & i \in I_{\mr{u}} \text{ and } i+1 \in I_{\mr{d}} \\
		1 & \text{otherwise.}
	\end{cases}
\end{equation*}
If we plug in the topological vertex formula \eqref{eq: TV formula} and use the identities
\begin{equation*}
	\kappa\big(\mu^{\transpose}\big) = - \kappa\big(\mu\big)\,, \qquad q^{-\kappa(\mu)/2}\, s_{\mu}\big(q^{\rho}\big) = s_{\mu^{\transpose}}\big(q^{-\rho}\big)
\end{equation*}
we get
\begin{align*}
	&\qquad\GWdiscarg{Z}{\gpT}{} =\\
	%& = \sum_{\substack{\mu_1,\ldots,\mu_{N-1} \\ \nu_1,\ldots,\nu_{N}}} \prod_{i=1}^{N-1}Q_i^{|\mu_i|} q_{i,i+1}^{-\kappa(\mu_i)/2} \prod_{i\in I_{\mr{u}}} q_4^{\kappa(\mu_{i})/2} s_{\frac{\mu_{i}^{\transpose}}{\nu_i}}(q_4^{\rho}) \, s_{\frac{\mu_{i-1}^{\transpose}}{\nu_i}}(q_4^{\rho}) ~ \prod_{i\in I_{\mr{d}}} q_5^{-\kappa(\mu_{i-1}^{\transpose})/2} s_{\frac{\mu_{i-1}}{\nu_i}}(q_5^{-\rho}) \, s_{\frac{\mu_{i}}{\nu_i}}(q_5^{-\rho}) \\
	& \sum_{\substack{\mu_1,\ldots,\mu_{N-1} \\ \nu_2,\ldots,\nu_{N-1}}} \prod_{i=1}^{N-1} Q_i^{|\mu_i|} s_{\mu_1}\!\big(q_{v_1}^{-\rho}\big) \left(\prod_{\substack{1<i<N\\i\in I_{\mr{u}}}} s_{\frac{\mu_{i-1}^{\transpose}}{\nu_i}}\!\big(q_4^{\rho}\big)\, s_{\frac{\mu_{i}^{\transpose}}{\nu_i}}\!\big(q_4^{\rho}\big)\right) \left(\prod_{\substack{1<i<N\\i\in I_{\mr{d}}}} s_{\frac{\mu_{i-1}}{\nu_i}}\big(q_5^{-\rho}\big) \, s_{\frac{\mu_{i}}{\nu_i}}\big(q_5^{-\rho}\big)\right) s_{\mu_{N-1}^{\transpose}}\!\big(q_{v_N}^{\rho}\big) 
\end{align*}
where we write $q_{v} = q_4$ for a vertex $v$ carrying an upwards pointing leg and $q_v=q_5$ otherwise. Following the approach of \cite{IKP06:VertexOnStrip}, one may evaluate the sum over partitions $\mu_i$, $\nu_i$ using the homogeneity of Schur functions $Q^{|\mu|-|\nu|} s_{\mu/\nu}(x) = s_{\mu/\nu}(Qx)$ and repeatedly applying the following specialisation of the skew Cauchy identities \cite[Sec.~I.5]{MacD95:SymFctns}:
\begin{align*}
	\sum_{\mu} s_{\frac{\mu}{\nu_1}} \big(Q q^{\rho}\big) ~ s_{\frac{\mu}{\nu_2}} \big(t^{\rho}\big) & = \Exp\left(\phantom{-}\frac{Q}{(q^{1/2} - q^{-1/2})(t^{1/2} - t^{-1/2})}\right) \cdot \sum_{\mu} s_{\frac{\nu_2}{\mu}} \big(Q q^{\rho}\big) ~ s_{\frac{\nu_1}{\mu}} \big(t^{\rho}\big) \,,\\
	\sum_{\mu} s_{\frac{\mu^{\transpose}}{\nu_1}} \big(Q q^{\rho}\big) ~ s_{\frac{\mu}{\nu_2}} \big(t^{\rho}\big) & = \Exp\left(-\frac{Q}{(q^{1/2} - q^{-1/2})(t^{1/2} - t^{-1/2})}\right) \cdot \sum_{\mu} s_{\frac{\nu_2^{\transpose}}{\mu}} \big(Q q^{\rho}\big) ~ s_{\frac{\nu_1^{\transpose}}{\mu^{\transpose}}} \big(t^{\rho}\big)\,.
\end{align*}
The resulting formula is
\begin{equation}
	\label{eq: strip formula}
	\GWdiscarg{Z}{\gpT}{} = \Exp \left(\sum_{1 \leq m \leq n \leq N} \frac{\prod_{k=m}^n Q_k}{\big(q_{v_m}^{1/2} - q_{v_m}^{-1/2}\big)\big(q_{v_n}^{1/2} - q_{v_n}^{-1/2}\big)}\right) \,.
\end{equation}
Conjectural formulas for strip geometries beyond \locsd\ torus action like \eqref{eq: A1 x C3 conj formula} will be presented in \cite{HS26:Experiments}.
\begin{rmk}
	\label{rmk: ref top vert strip}
	In general, the above expression agrees with formulae produced via the refined topological vertex only when the moduli space of stable maps to $X$ is proper in all genera and curve classes. This is for instance the case for the resolved conifold. For non-proper the moduli spaces the quantities generally disagree.
\end{rmk}

\begin{rmk}
	\label{rmk: gauge th only sd limit}
	A toric Calabi--Yau threefold $X$ engineering supersymmetric $\mr{SU}(N)$ gauge theory on $\Aaff[2]$ may be obtained by gluing two strip geometries with all legs pointing upwards resp.~downwards along the vertical non-compact directions. So from the above discussion the reader might be tempted to hope that refined invariants of such an $X$ (that is invariants on the so-called general $\Omega$ background $\epsilon_4$, $\epsilon_5$) may be computed via our vertex formalism. This is, however, unfortunately impossible because for the gluing to be compatible with the torus action one has to impose the constraint $\epsilon_2 = -\epsilon_3$ which in turn forces $\epsilon_4 = -\epsilon_5$. This means the vertex formalism presented in this note cannot compute the Gromov--Witten invariants of $X\times \Aaff[2]$ beyond the self-dual limit which is already well-studied in the literature \cite{IKP06:GaugeStringDuality}.
\end{rmk}

\medskip

\subsection{The GW dual of rank-two DT theory on the resolved conifold}
\label{sec: RC x A1}
Let us consider our first example of a fivefold featuring a compact four cycle:
\begin{equation*}
	Z = \Tot{}_{\bP^1 \times \bP^1} \big(\cO(-1,0) \oplus \cO(-1,0) \oplus \cO(0,-2)\big) \longrightarrow \bP^1 \times \bP^1
\end{equation*}
This fivefold is the product of the resolved conifold and the resolution of the $A_1$ surface singularity.

We assume that the dense torus $\widehat{\gpT}\cong \Gm[5]$ acts with tangent weights $\epsilon_1$, $\epsilon_2$ at the fixed point $(0,0)\in\bP^1 \times \bP^1$ and on the fibre of each line bundle over $(0,0)$ with tangent weight $\epsilon_3$, $\epsilon_4$ and $\epsilon_5$ respectively. We denote by $\gpT$ the two-dimensional subtorus of $\widehat{\gpT}$ for which
\begin{equation*}
	\epsilon_3 = - \epsilon_2 \,,\qquad \epsilon_4 = \epsilon_2 \,,\qquad \epsilon_5 = -\epsilon_1-\epsilon_2\,.
\end{equation*}
\Cref{fig: RC x A1} illustrates the resulting $\gpT$-diagram. Observe that this torus action is both Calabi--Yau and \locsd.

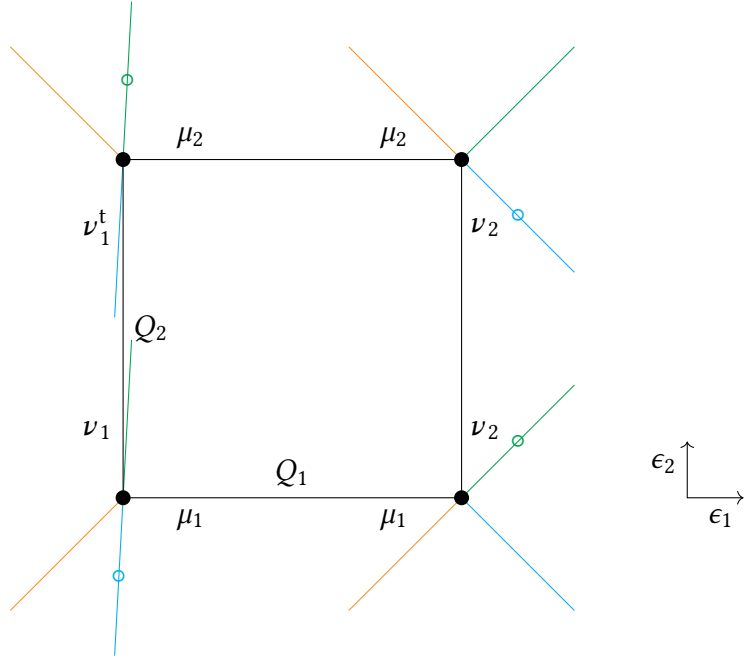
\begin{figure}
	\centering
	\begin{tikzpicture}[scale=1.5]
		\draw[Cyan] (0,0) -- (-0.075,-1.4);
		\node[Cyan] at ($(0,0)!0.5!(-0.075,-1.4)$) {$\circ$};
		\draw[Cyan] (0,3) -- (-0.075,1.6);
		\draw[Cyan] (3,0) -- (4,-1);
		\draw[Cyan] (3,3) -- (4,2);
		\node[Cyan] at ($(3,3)!0.5!(4,2)$) {$\circ$};
		\draw[Green] (0,0) -- (0.075,1.4);
		\draw[Green] (0,3) -- (0.075,4.4);
		\node[Green] at ($(0,3)!0.5!(0.075,4.4)$) {$\circ$};
		\draw[Green] (3,0) -- (4,1);
		\node[Green] at ($(3,0)!0.5!(4,1)$) {$\circ$};
		\draw[Green] (3,3) -- (4,4);
		\draw[BurntOrange] (0,0) -- (-1,-1);
		\draw[BurntOrange] (0,3) -- (-1,4);
		\draw[BurntOrange] (3,0) -- (2,-1);
		\draw[BurntOrange] (3,3) -- (2,4);
		\draw (0,0) -- (3,0) -- (3,3) -- (0,3) -- (0,0);
		\node[below] at ($(0,0)!0.2!(3,0)$) {$\mu_1$};
		\node[below] at ($(0,0)!0.8!(3,0)$) {$\mu_1$};
		\node[right] at ($(3,0)!0.2!(3,3)$) {$\nu_2$};
		\node[right] at ($(3,0)!0.8!(3,3)$) {$\nu_2$};
		\node[above] at ($(3,3)!0.2!(0,3)$) {$\mu_2$};
		\node[above] at ($(3,3)!0.8!(0,3)$) {$\mu_2$};
		\node[left] at ($(0,0)!0.2!(0,3)$) {$\nu_1$};
		\node[left] at ($(0,0)!0.8!(0,3)$) {$\nu_1^{\transpose}$};
		\node[above] at (1.5,0) {$Q_1$};
		\node[right] at (0,1.5) {$Q_2$};
		\node[below] at (5.3,0) {$\epsilon_1$};
		\node[left] at (5,0.3) {$\epsilon_2$};
		\draw[<->] (5,0.5) -- (5,0) -- (5.5,0);
		\node[circle,inner sep=2pt,fill=black] at (0,0){};
		\node[circle,inner sep=2pt,fill=black] at (3,0){};
		\node[circle,inner sep=2pt,fill=black] at (0,3){};
		\node[circle,inner sep=2pt,fill=black] at (3,3){};
	\end{tikzpicture}
	\caption{The $\gpT$-diagram of $\Tot \cO(-1,0) \oplus \cO(-1,0) \oplus \cO(0,-2)$ with a choice of \distdir\ at each vertex and half-edges decorated by partitions. The half-edges associated with a line bundle are coloured blue, green and orange respectively. All vertical lines on the left should be parallel. The tilt of the coloured half-edges is solely for display purposes.}
	\label{fig: RC x A1}
\end{figure}

To apply \Cref{thm: vertex formalism main part} we pick \distdir s at each vertex as indicated in \Cref{fig: RC x A1}. We order the remaining half-edges clockwise for the bottom and anti-clockwise for the top vertices. The resulting \mxp s may be inferred from \Cref{fig: RC x A1} from how we decorate half-edges by partitions. In this situation \Cref{cor: vertex formalism curve class} provides us with the following formula:
\begin{equation*}
	\begin{split}
		& \qquad \GWdiscarg{Z}{\gpT}{}= \\[0.3em]
		&\sum_{\substack{\mu_1,\mu_2 \\ \nu_1,\nu_2}} Q_1^{|\mu_1|+|\mu_2|} Q_2^{|\nu_1|+|\nu_2|} (-1)^{|\nu_1|} \mathcal{W}_{\nu_1,\mu_1,\varnothing}\big(q_2^{-1}\big) ~%
		\mathcal{W}_{\mu_1,\nu_2,\varnothing}\big(q_1 q_2\big) ~ %
		\mathcal{W}_{\mu_2,\nu_2,\varnothing}\big(q_1 q_2^{-1}\big) ~ %
		\mathcal{W}_{\nu_1^{\transpose},\mu_2,\varnothing}\big(q_2\big)\,.
	\end{split}
\end{equation*}
Based on computer experiments we expect that the above sum over partitions can be carried out explicitly to yield the following formula.
\begin{conj}
	\label{conj: closed form RC times A1}
	We have
	\begin{equation*}
		\GWdiscarg{Z}{\gpT}{} = \Exp \left(\frac{\big(q_1^{1/2} + q_1^{3/2}\big)\, Q_1}{\big(1-q_1 q_2\big)\big(1-q_1 q_2^{-1}\big)}\right) \cdot \Exp \left( \frac{\big(1-q_1\big)^2 \, q_2 \, Q_2}{\big(1-q_2\big)^2\big(1-q_1q_2\big)\big(1-q_1 q_2^{-1}\big)}\right)\,.
	\end{equation*}
\end{conj}

\begin{rmk}
	The above observed factorisation into a contribution coming from the resolved conifold and another coming from the resolution of the $A_1$ singularity without the presence of cross-terms conjecturally occurs in more a general situation: Such a factorisation should happen for all products of the form $X \times \cA_r$ where $X$ is a Calabi--Yau threefold and $\cA_r$ is the resolution of the $A_r$ surface singularity. When the torus action on $\cA_r$ is Calabi--Yau the absence of cross-terms is a consequence of the vanishing of the virtual fundamental class due to the nowhere vanishing holomorphic two-form on the surface. Numerical evidence beyond Calabi--Yau torus actions will be presented in \cite{HS26:Experiments}.
\end{rmk}

%\begin{rmk}
%	Conceptionally, the factorisation can be motivated as follows: Modulo perturbative corrections (which can be identified with the $\cA_r$-factor) the M2-brane index on $X \times \cA_r$ engineers rank-$r$ Donaldson--Thomas theory on $X$ \cite[Sec.~5.4 \& 5.5]{NO14:membranes} where K\"{a}hler variables of $\cA_r$ being identified with $\mathop{GL}(r)$-equi\-vari\-ant parameters of the Donaldson--Thomas invariants. The absence of cross-terms is then a consequence of the independence of the latter invariants of these equivariant parameters \cite[Eq.~(5.39)]{dZNPZ22:PlayingIndexMth}.
%\end{rmk}

\medskip

\subsection{\texorpdfstring{$\boldsymbol{\Tot{}_{\bP^2}\big( \cO(-1)^{\oplus 3}\big)}$}{Tot O(-1)\^{}3 -> P2}}
\label{sec: O minus 1 plus 3 on P2}
In this section we consider the fivefold
\begin{equation*}
	Z = \Tot{}_{\bP^2}\big(\cO_{}(-1) \oplus \cO_{}(-1) \oplus \cO_{}(-1)\big) \longrightarrow \bP^2
\end{equation*}
with respect to a specific torus action. The easiest way to describe it is by presenting $Z$ as a quotient
\begin{equation*}
	Z = \Aaff[6] \setminus V(x_0 x_1 x_2) \big/ \Gm 
\end{equation*}
where the torus $\Gm$ acts via
\begin{equation*}
	\big(t,(x_0,x_1,x_2,y_0,y_1,y_2)\big) \longmapsto ( t x_0, t x_1, t x_2, t^{-1} y_0, t^{-1} y_1, t^{-1} y_2 )
\end{equation*}
on affine space. The natural torus action of $\Gm[6]$ on $\Aaff[6]$, whose tangent weights we denote by $\epsilon_0,\epsilon_1,\epsilon_2,\alpha_0,\alpha_1,\alpha_2$, descends to an action on the quotient $Z$. At the fixed point $[1:0:0]\in\bP^2\subset Z$ the tangent weights read
\begin{equation*}
	\epsilon_1 - \epsilon_0\,, \qquad \epsilon_2 - \epsilon_0 \,, \qquad \alpha_0 + \epsilon_0\,, \qquad \alpha_1 + \epsilon_0\,, \qquad \alpha_2 + \epsilon_0
\end{equation*}
and similar for the two other fixed points. We denote by $\gpT\cong \Gm[3]$ the subtorus of $\Gm[6]$ where the following relations hold:
\begin{equation*}
	\alpha_0 = -\epsilon_0 + \epsilon_1 - \epsilon_2 \,, \qquad \alpha_1 = -\epsilon_0 - \epsilon_1 + \epsilon_2  \,, \qquad \alpha_2 = \epsilon_0 - \epsilon_1 - \epsilon_2\,.
\end{equation*}
One checks that the action of $\gpT$ on $Z$ is Calabi--Yau and \locsd\ as illustrated in \Cref{fig: O minus 1 plus 3 on P2}. We choose the \distdir s as indicated in the figure and orient all half-edges at the vertices clockwise. %
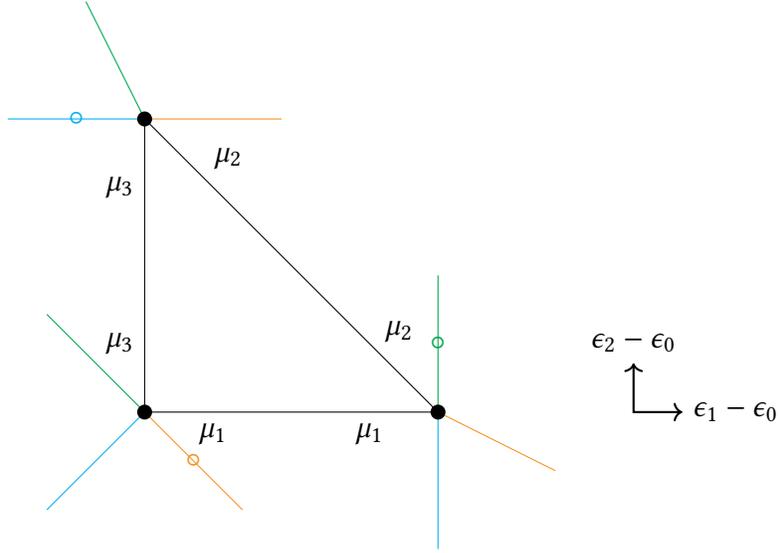
\begin{figure}%
	\centering
	\begin{tikzpicture}[scale=1.3]
		\draw[Cyan] (0,0) -- (-1,-1);
		\draw[Cyan] (0,3) -- (-1.4,3);
		\node[Cyan] at (-0.7,3) {$\circ$};
		\draw[Cyan] (3,0) -- (3,-1.4);
		\draw[Green] (0,0) -- (-1,1);
		\draw[Green] (0,3) -- (-0.6,4.2);
		\draw[Green] (3,0) -- (3,1.4);
		\node[Green] at (3,0.7) {$\circ$};
		\draw[BurntOrange] (0,0) -- (1,-1);
		\node[BurntOrange] at (0.5,-0.5) {$\circ$};
		\draw[BurntOrange] (0,3) -- (1.4,3);
		\draw[BurntOrange] (3,0) -- (4.2,-0.6);
		\draw (0,0) -- (3,0) -- (0,3) -- (0,0);
		\node[circle,inner sep=2pt,fill=black] at (0,0){};
		\node[circle,inner sep=2pt,fill=black] at (3,0){};
		\node[circle,inner sep=2pt,fill=black] at (0,3){};
		\node[right] at (5.5,0) {$\epsilon_1-\epsilon_0$};
		\node[above] at (5,0.5) {$\epsilon_2-\epsilon_0$};
		\draw[<->,thick] (5,0.5) -- (5,0) -- (5.5,0);
		\node[below] at (0.7,0) {$\mu_1$};
		\node[below] at (2.3,0) {$\mu_1$};
		\node[above] at (2.6,0.6) {$\mu_2$};
		\node[right] at (0.6,2.6) {$\mu_2$};
		\node[left] at (0,0.7) {$\mu_3$};
		\node[left] at (0,2.3) {$\mu_3$};		
	\end{tikzpicture}
	\caption{The $\gpT$-diagram of $\Tot \cO_{\bP^2}(-1)^{\oplus 3}$ with a choice of a \distdir\ at each vertex and half-edges decorated by partitions.}
	\label{fig: O minus 1 plus 3 on P2}
\end{figure}%
Applying \Cref{cor: vertex formalism curve class} to this setup yields the formula
\begin{equation}
	\label{eq: formula Omin1 plus 3 P2}
\begin{split}
	\GWdiscarg{Z}{\gpT}{} = \sum_{\mu_1, \mu_2, \mu_3}   (-Q)^{|\mu_1|+|\mu_2|+|\mu_3|} \big(q_1^{-1} q_2\big)^{\kappa(\mu_1)/2} \big(q_2^{-1} q_0\big)^{\kappa(\mu_2)/2}  \big(q_0^{-1}q_1\big)^{\kappa(\mu_3)/2}&  \\[-0.3em]
	\times \cW_{\mu_1,\mu_2,\varnothing}\big(q_2 q_0^{-1}\big) ~
	\cW_{\mu_2,\mu_3,\varnothing}\big(q_0 q_1^{-1}\big) ~ %
	\cW_{\mu_3,\mu_1,\varnothing}\big(q_1 q_2^{-1}\big) & \,. %
\end{split}
\end{equation}
The author is not aware of any trick that allows one to carry out the above sum explicitly. However, one may still use the formula to determine $\Omega_{d[H]}$ in low degree $d$. As formula \eqref{eq: formula Omin1 plus 3 P2} inherited the full $\mf{S}_3$ Weyl symmetry of $\bP^2$ we may expand the \MemInds\ in terms of elementary symmetric polynomials:
\begin{equation*}
	e_1 = q_0^{1/2} + q_1^{1/2} + q_2^{1/2}\,, \qquad e_2 = (q_0 q_1)^{1/2} + (q_0 q_2)^{1/2} + (q_1 q_2)^{1/2}\,, \qquad e_3 = (q_0 q_1 q_2)^{1/2}\,.
\end{equation*}
The resulting expressions become particularly nice when normalised by the symmetrised $q$-number
\begin{equation*}
	[n]_q \coloneqq \frac{q^{n/2} - q^{-n/2}}{q^{1/2} - q^{-1/2}} \,.
\end{equation*}
Moreover, let us write $\overline{q}_i \coloneqq q_{i}^{1/2} q_{i+1}^{-1/2}$. With this notation the \MemInds\ extracted from \eqref{eq: formula Omin1 plus 3 P2} in low degree read:
{\allowdisplaybreaks%
\begin{align*}
	\Omega_{1[H]} & = - \prod_{i=0}^2\frac{1}{[2]_{\overline{q}_i} }  \\[0.5em]
	\Omega_{2[H]} & = \prod_{i=0}^2\frac{1}{[2]_{\overline{q}_i} }  \\[0.5em]
	\Omega_{3[H]} & = \prod_{i=0}^2\frac{1}{[2]_{\overline{q}_i } [4]_{\overline{q}_i} } \\
	& \qquad \times \left(-e_1^4 e_2^4 e_3^{-4}+3 e_1^2 e_2^5 e_3^{-4}-e_2^6 e_3^{-4}+3 e_1^5 e_2^2 e_3^{-3}-8 e_1^3 e_2^3 e_3^{-3}-e_1^6 e_3^{-2} +11 e_1^2 e_2^2 e_3^{-2} \right. \\
	& \qquad \hspace{1.3em}\left.-3 e_2^3 e_3^{-2} -3 e_1^3 e_3^{-1}\right) \\[0.5em]
	\Omega_{4[H]} & = \prod_{i=0}^2\frac{[3]_{\overline{q}_i} }{ [2]_{\overline{q}_i} [4]_{\overline{q}_i} [6]_{\overline{q}_i} } \\
	& \qquad \times\left(e_1^9 e_2^9 e_3^{-9}-8 e_1^7 e_2^{10} e_3^{-9}+21 e_1^5 e_2^{11} e_3^{-9}-19 e_1^3 e_2^{12} e_3^{-9}+3 e_1 e_2^{13} e_3^{-9}-8 e_1^{10} e_2^7 e_3^{-8} \right.\\ 
	&\qquad \hspace{1.3em} +63 e_1^8 e_2^8 e_3^{-8}-155 e_1^6 e_2^9 e_3^{-8}+109 e_1^4 e_2^{10} e_3^{-8}+15 e_1^2 e_2^{11} e_3^{-8}-2 e_2^{12} e_3^{-8} \\
	&\qquad \hspace{1.3em} +21 e_1^{11} e_2^5 e_3^{-7}-155 e_1^9 e_2^6 e_3^{-7}+295 e_1^7 e_2^7 e_3^{-7}+56 e_1^5 e_2^8 e_3^{-7}-332 e_1^3 e_2^9 e_3^{-7}\\
	&\qquad \hspace{1.3em} +8 e_1 e_2^{10} e_3^{-7}-19 e_1^{12} e_2^3 e_3^{-6}+109 e_1^{10} e_2^4 e_3^{-6}+56 e_1^8 e_2^5 e_3^{-6}-958 e_1^6 e_2^6 e_3^{-6}\\
	&\qquad \hspace{1.3em} +816 e_1^4 e_2^7 e_3^{-6}+294 e_1^2 e_2^8 e_3^{-6}-3 e_2^9 e_3^{-6}+3 e_1^{13} e_2 e_3^{-5}+15 e_1^{11} e_2^2 e_3^{-5}\\
	&\qquad \hspace{1.3em} -332 e_1^9 e_2^3 e_3^{-5}+816 e_1^7 e_2^4 e_3^{-5}+492 e_1^5 e_2^5 e_3^{-5}-1356 e_1^3 e_2^6 e_3^{-5}-129 e_1 e_2^7 e_3^{-5}\\
	&\qquad \hspace{1.3em} -2 e_1^{12} e_3^{-4}+8 e_1^{10} e_2 e_3^{-4}+294 e_1^8 e_2^2 e_3^{-4}-1356 e_1^6 e_2^3 e_3^{-4}+554 e_1^4 e_2^4 e_3^{-4}\\
	&\qquad \hspace{1.3em} +950 e_1^2 e_2^5 e_3^{-4}+35 e_2^6 e_3^{-4}-3 e_1^9 e_3^{-3}-129 e_1^7 e_2 e_3^{-3}+950 e_1^5 e_2^2 e_3^{-3}\\
	&\qquad \hspace{1.3em} -690 e_1^3 e_2^3 e_3^{-3}-374 e_1 e_2^4 e_3^{-3}+35 e_1^6 e_3^{-2}-374 e_1^4 e_2 e_3^{-2}+305 e_1^2 e_2^2 e_3^{-2}\\
	&\qquad \hspace{1.3em} \left.+70 e_2^3 e_3^{-2}+70 e_1^3 e_3^{-1}-60 e_1 e_2 e_3^{-1}\right)
\end{align*}}

\begin{rmk} \label{rmk: features of O min 1 oplus 3 on P2}
	Let us emphasise the following remarkable features of the above formulae.
	\begin{enumerate}
		\item First, observe that all formulae are in agreement with \Cref{conj: generalised GV}: All expressions are elements in localised equivariant K-theory with integer coefficients.
		
		\item However, in the limit $\overline{q}_i \rightarrow 1$ coefficients in the above formulae feature negative powers of $2$. This is in accordance with \Cref{conj: generalised GV} since this is precisely the limit in which the torus action becomes non-skeletal.
		
		\item It should also be stressed that powers of $2$ are indeed the worst denominators that appear. The cancellations ensuring this are surprisingly fine-tuned and based on numerical data we conjecture the following general behaviour in higher degree.
		
%		\item The formulae feature half-integer powers of the representations $q_i$. This is in accordance with the comment in \Cref{conj: generalised GV} that one may have to pass to a cover of the original torus in order to ensure that a lift from Cohomology to K-theory indeed exists.
	\end{enumerate}
\end{rmk}

\begin{conj}
	\label{conj: O minus 1 plus 3 on P2}
	For any $d > 1$ we have
	\begin{equation*}
		\left(\prod_{i=0}^2 \prod_{n=1}^{d-1}\frac{ [2n]_{\overline{q}_i} }{[ \mr{od}(n)]_{\overline{q}_i}}\right) \cdot \Omega_{d[H]} \in \bZ[e_1^{\pm 1}, e_2,e_3]
	\end{equation*}
	where $\mr{od}(n)$ denotes the odd part of an integer $n$.
\end{conj}

We checked this conjecture numerically up to degree ten. Conjectural formulae for low degree $\Omega_{\beta}$ on the full four-dimensional Calabi--Yau torus of $Z$ will be presented in \cite{HS26:Experiments}.

\medskip

\subsection{\texorpdfstring{$\boldsymbol{\Tot{}_{\bP^3}\big( \cO(-2)^{\oplus 2}\big)}$}{Tot O(-2)\^{}2 -> P3}}
\label{sec: O min 2 oplus 2 on P3}
In this section we discuss
\begin{equation*}
	Z = \Tot{}_{\bP^3}\big( \cO(-2) \oplus \cO(-2)\big) \longrightarrow \bP^3 \,.
\end{equation*}
As in the last section we present this variety as a quotient
\begin{equation*}
	Z = \Aaff[6] \setminus V(x_0 x_1 x_2 x_3) \big/ \Gm
\end{equation*}
where the torus acts on affine space via
\begin{equation*}
	\big(t,(x_0,x_1,x_2,x_3,y_0,y_1)\big) \longmapsto ( t x_0, t x_1, t x_2, t x_3, t^{-2} y_0, t^{-2} y_1 )\,.
\end{equation*}
The action of $\Gm[6]$ on $\Aaff[6]$, whose tangent weights we denote by $\epsilon_0,\epsilon_1,\epsilon_2, \epsilon_3,\alpha_0,\alpha_1$, descends to the quotient $Z$. We will analyse the equivariant Gromov--Witten theory of $Z$ with respect to a four-dimensional torus $\gpT\subset \Gm[6]$ which is subject to the constraints $\alpha_0=-\epsilon_1$ and $\alpha_1 = \epsilon_0-\epsilon_2-\epsilon_3$. One can check that this torus action is Calabi--Yau and \locsd. Indeed, for instance at the fixed point $[1:0:0:0]\in \bP^3\subset Z$ we find the tangent weights
\begin{equation*}
	\epsilon_1 - \epsilon_0\,, \qquad \epsilon_2 - \epsilon_0\,, \qquad \epsilon_3 - \epsilon_0\,, \qquad -\epsilon_1 + \epsilon_0 \,,\qquad 2 \epsilon_0 -\epsilon_2 -\epsilon_3\,.
\end{equation*}
The $\gpT$-diagram of $Z$ which is displayed in \Cref{fig: T diagram O min 2 oplus 2 on P3} indicates the tangent weights at the remaining fixed points.
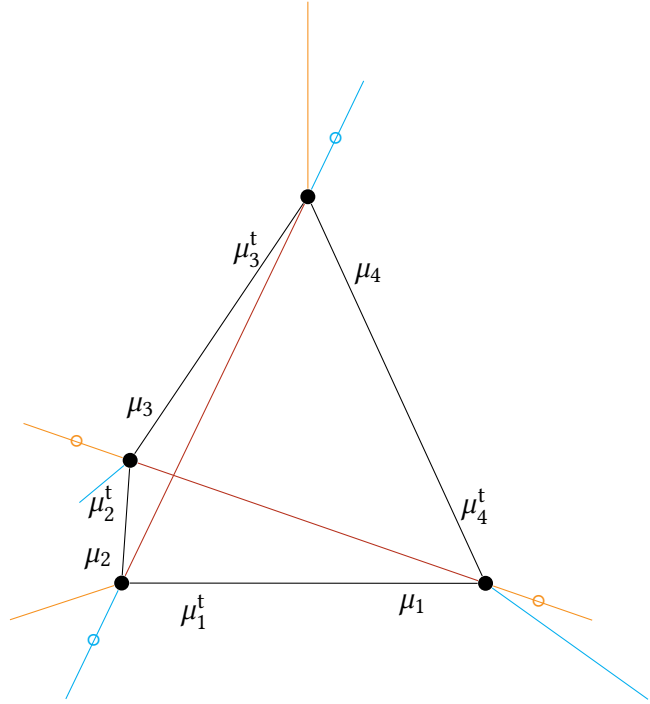
\begin{figure}
\centering
\begin{tikzpicture}[x={({cos(19.1)},{-sin(19.1)},0)},z={({-sin(40)},{-cos(40)},0)},scale=2.5]
	\foreach \pos [count=\Ind] in {{0.00, 0.00, 0.00},{2.00, 0.00, 0.00},{1.00, 1.73, 0.00},{1.00, 0.58, 1.63}}{
		\node[circle,inner sep=2pt,fill=black](p\Ind) at (\pos){};
	}
	%
%	\draw[thick, yellow] (p1) -- (p3) -- (p2) -- (p4) -- (p1);
	\draw (p1) -- (p3) -- (p2) -- (p4) -- (p1);
%	\draw[line width=3pt,red] (p1) -- (p2);
%	\draw[line width=3pt,red] (p3) -- (p4);
	\draw[BrickRed] (p1) -- (p2);
	\draw[BrickRed] (p3) -- (p4);
	\node[below] at ($(p2)!0.2!(p4)$) {$\mu_1$};
	\node[below] at ($(p2)!0.8!(p4)$) {$\mu_1^{\transpose}$};
	\node[left] at ($(p4)!0.2!(p1)$) {$\mu_2$};
	\node[left] at ($(p4)!0.65!(p1)$) {$\mu_2^{\transpose}$};
	\node[left] at ($(p1)!0.2!(p3)$) {$\mu_3$};
	\node[left] at ($(p1)!0.8!(p3)$) {$\mu_3^{\transpose}$};
	\node[right] at ($(p3)!0.2!(p2)$) {$\mu_4$};
	\node[right] at ($(p3)!0.8!(p2)$) {$\mu_4^{\transpose}$};
	\draw[Cyan] (p3) -- (1., 2.075, -0.489);
	\node[Cyan] at (1., 1.9025, -0.2445) {$\circ$};
	\draw[Cyan] (p4) -- (1., 0.235, 2.119);
	\node[Cyan] at (1., 0.4075, 1.8745) {$\circ$};
	\draw[BurntOrange] (p1) -- (-0.6, 0., 0.);
	\node[BurntOrange] at (-0.3, 0., 0.) {$\circ$};
	\draw[BurntOrange] (p2) -- (2.6, 0., 0.);
	\node[BurntOrange] at (2.3, 0., 0.) {$\circ$};
	\draw[Cyan] (p1) -- (-0.6, -0.693, -0.489);
	\draw[Cyan] (p2) -- (2.6, -0.693, -0.489);
	\draw[BurntOrange] (p3) -- (1., 2.768, 0.);
	\draw[BurntOrange] (p4) -- (1., 0.928, 2.608);
\end{tikzpicture}
\caption{The $\gpT$-diagram of $\Tot \cO_{\bP^3}(-2)^{\oplus 2}$ with a choice of a \distdir\ at each vertex and half-edges decorated by partitions.}
\label{fig: T diagram O min 2 oplus 2 on P3}
\end{figure}

Now note that despite the fact that the $\gpT$-action on $Z$ is skeletal, Calabi--Yau and \locsd\ we cannot readily apply \Cref{cor: vertex formalism curve class} to compute the Gromov--Witten invariants of $Z$  --- at least not in the form it is stated. The problem is that the class of a line can be supported on both of the two \sd\ edges of the $\gpT$-diagram which are highlighted red in \Cref{fig: T diagram O min 2 oplus 2 on P3}. However, as explained in \Cref{rmk: weaker condition curve class} the conclusion of \Cref{cor: vertex formalism curve class} still holds if we can show that $$\GWdiscarg{Z}{\gpT}{\boldsymbol{d}}=0$$ whenever $\boldsymbol{d}$ has non-trivial support on one of the two red edges. So let $F_\gamma$ be a component of $\gpT$-fixed locus of $\Mbar(Z,\beta)$ parametrising stable maps with non-zero support on one of the red edges. It suffices to show that $[F_\gamma]^{\vir}=0$. Indeed, the restriction of the vector bundle $Z \rightarrow \bP^3$ to each of the red edges is isomorphic to $\cO_{\bP^1}(-2) \oplus \cO_{\bP^1}(-2)$. Moreover, $\mgp{T}$ acts on one of the line bundles in a way that the holomorphic two-form of $\Tot{}_{\bP^1} ( \cO(-2) )$ is fixed. This nowhere vanishing $\mgp{T}$-invariant holomorphic two-form thus yields a trivial factor in the obstruction bundle of $F_\gamma$. This implies the vanishing of the virtual class as desired.

Hence, we may apply the conclusion of \Cref{cor: vertex formalism curve class} to our case at hand. Note that by the vanishing we have just shown we can decorate the red edges with trivial partitions as indicated in \Cref{fig: T diagram O min 2 oplus 2 on P3}. We fix a cyclic order at each vertex $v$ by demanding that $\sigma_v$ maps the half-edge decorated with $\mu_i$ to $\mu_{i+1}$. Together with the choice of \distdir s indicated in \Cref{fig: T diagram O min 2 oplus 2 on P3} this implies that all edges have negative \mxp . Let us write $Q\coloneqq Q^{[L]}$ where $[L]$ is the class of a line in $\bP^3$. We obtain the following formula for the Gromov--Witten invariants of $Z$:
\begin{equation}
\label{eq: formula O min 2 oplus 2 on P3}
\begin{split}
	&\qquad \GWdiscarg{Z}{\gpT}{} = \\[0.3em]
	&\sum_{\substack{\mu_1,\mu_2,\\\mu_3,\mu_4}} Q^{\sum_{i=1}^4|\mu_i|}  \big(q_0^{-2}q_1^2 q_2^{-1} q_3\big)^{\kappa(\mu_1)/2} \big(q_0 q_1^{-1} q_2^{-2} q_3^2\big)^{\kappa(\mu_2)/2}  \big(q_0^{2}q_1^{-2} q_2 q_3^{-1}\big)^{\kappa(\mu_3)/2} \big(q_0^{-1} q_1 q_2^{2} q_3^{-2}\big)^{\kappa(\mu_4)/2} \\[-0.8em]
	& \hspace{4.5em} \times \cW_{\mu_4^{\transpose},\mu_1,\varnothing}\big(q_0 q_1^{-1}\big) ~ \cW_{\mu_1^{\transpose},\mu_2,\varnothing}\big(q_2 q_3^{-1}\big) ~ \cW_{\mu_2^{\transpose},\mu_3,\varnothing}\big(q_0^{-1} q_1\big) ~ \cW_{\mu_3^{\transpose},\mu_4,\varnothing}\big(q_2^{-1} q_3\big) \,. 
\end{split}
\end{equation}

As in the last example we are not able to carry out the above sums explicitly but we may still employ the formula to extract \MemInds\ in low degree. Experimentally we observe the following.

\begin{conj}
	\label{conj: O min 2 oplus 2 on P3}
	$\Omega_{d[L]}$ is a Laurent polynomial in $q_0$, $q_1$, $q_2$ and $q_3$.
\end{conj}

Moreover, observe that formula \eqref{eq: formula O min 2 oplus 2 on P3} only depends on the characters $q_0 q_1^{-1}$ and $q_2 q_3^{-1}$ and is invariant under $q_0 q_1^{-1} \rightarrow q_0^{-1} q_1$ and $q_2 q_3^{-1} \rightarrow q_2^{-1} q_3$. The latter symmetry is inherited from the $\bZ/2\bZ \times \bZ/2\bZ$ symmetry of \Cref{fig: T diagram O min 2 oplus 2 on P3}. Hence, assuming the absence of poles other than zero or infinity, the \MemInds\ can be expanded as
\begin{equation*}
	\Omega_{d[L]} \eqqcolon \sum_{k_1,k_2\geq 1} N_{d;k_1,k_2} \, [k_1]_{q_0q_1^{-1}} \, [k_2]_{q_2q_3^{-1}}
\end{equation*}
with $N_{d;k_1,k_2}\in\bZ$. In degree $d\leq 3$ the only non-vanishing invariants are:
\begin{equation*}
	N_{2;1,1} = 2 \, , \qquad N_{3;4,4} = -2\,.
\end{equation*}
All non-zero invariants for $4\leq d \leq 6$ are listed in \Crefrange{tab: d4 GV O min 2 oplus 2 on P3}{tab: d6 GV O min 2 oplus 2 on P3}. We checked that \Cref{conj: O min 2 oplus 2 on P3} holds up to degree ten numerically. Conjectural formulae for low-degree \MemInds\ on the full four dimensional Calabi--Yau torus of $Z$ will be presented in \cite{HS26:Experiments}.
\begin{table}[H]
	\centering
	\begin{minipage}{.4\linewidth}
	\centering
	\begin{tabular}{c|ccccc}
			& 1 & 3 & 5 & 7 & 9 \\ \hline
			1 & -1 &   & 2 &   & 2 \\
			3 &   & 2 &   &   &   \\
			5 & 2 &   & 2 &   & 2 \\
			7 &   &   &   &   &   \\
			9 & 2 &   & 2 &   & 2 
	\end{tabular}
	\caption{$N_{4;k_1,k_2}$}
	\label{tab: d4 GV O min 2 oplus 2 on P3}
	\end{minipage}%
	\begin{minipage}{.6\linewidth}
	\centering
	\begin{tabular}{c|cccccccc}
		& 2 & 4 & 6 & 8 & 10 & 12 & 14 & 16 \\ \hline
		2 & -2 &   & -2 & -2 &   &   &   &   \\
		4 &   & 2 & -2 & -2 & -2 & -2 &   & -2 \\
		6 & -2 & -2 & -4 & -4 & -2 & -2 &   & -2 \\
		8 & -2 & -2 & -4 & -4 & -2 & -2 &   & -2 \\
		10 &   & -2 & -2 & -2 & -2 & -2 &   & -2 \\
		12 &   & -2 & -2 & -2 & -2 & -2 &   & -2 \\
		14 &   &   &   &   &   &   &   &   \\
		16 &   & -2 & -2 & -2 & -2 & -2 &   & -2
	\end{tabular}
	\caption{$N_{5;k_1,k_2}$}
	\label{tab: d5 GV O min 2 oplus 2 on P3}
\end{minipage}%
\end{table}

\begin{table}[H]
	\centering
	\begin{tabular}{c|ccccccccccccc}
		& 1 & 3 & 5 & 7 & 9 & 11 & 13 & 15 & 17 & 19 & 21 & 23 & 25 \\ \hline
		1 & 7 & -3 & 7 & 1 & 10 & 4 & 12 & 4 & 8 & 4 & 4 &   & 4 \\
		3 & -3 & -3 & -1 & 3 & 2 & 2 & 2 & 2 &   &   &   &   &   \\
		5 & 7 & -1 & 13 & 5 & 14 & 6 & 14 & 6 & 8 & 4 & 4 &   & 4 \\
		7 & 1 & 3 & 5 & 9 & 10 & 6 & 10 & 6 & 4 & 2 & 2 &   & 2 \\
		9 & 10 & 2 & 14 & 10 & 18 & 8 & 20 & 8 & 12 & 6 & 6 &   & 6 \\
		11 & 4 & 2 & 6 & 6 & 8 & 4 & 8 & 4 & 4 & 2 & 2 &   & 2 \\
		13 & 12 & 2 & 14 & 10 & 20 & 8 & 20 & 8 & 12 & 6 & 6 &   & 6 \\
		15 & 4 & 2 & 6 & 6 & 8 & 4 & 8 & 4 & 4 & 2 & 2 &   & 2 \\
		17 & 8 &   & 8 & 4 & 12 & 4 & 12 & 4 & 8 & 4 & 4 &   & 4 \\
		19 & 4 &   & 4 & 2 & 6 & 2 & 6 & 2 & 4 & 2 & 2 &   & 2 \\
		21 & 4 &   & 4 & 2 & 6 & 2 & 6 & 2 & 4 & 2 & 2 &   & 2 \\
		23 &   &   &   &   &   &   &   &   &   &   &   &   &   \\
		25 & 4 &   & 4 & 2 & 6 & 2 & 6 & 2 & 4 & 2 & 2 &   & 2 
	\end{tabular}
	\caption{$N_{6;k_1,k_2}$}
	\label{tab: d6 GV O min 2 oplus 2 on P3}
\end{table}

%-------------------------------------------------------------------------------
\section{Proof of the vertex formalism}
\label{sec: proof vertex formalism}
In this \namecref{sec: proof vertex formalism} we will prove \Cref{thm: vertex formalism main part}. Throughout this section let $Z$ be a Calabi--Yau fivefold with a skeletal, Calabi--Yau and \locsd\ action by a torus $\gpT$. Moreover, we fix a choice of \distdir s and orders. We will deduce the vertex formula for Gromov--Witten invariants of $Z$ from capped localisation. This method repackages the formula for Gromov--Witten invariants resulting from torus localisation in a more practical form. This idea was pioneered by Li--Liu--Liu--Zhou in \cite{LLLZ09:MathTopVert} (see also \cite{LLZ03:MarinoVafaFormula,LLZ07:2PartnHodgeInt}) and was key in Maulik--Oblomkov--Okunkov--Pandharipande's proof of the Gromov--Witten/Donaldson--Thomas correspondence for toric threefolds \cite{MOOP11:GWDTtoric}.

\subsection{Torus localisation}
Recall from \Cref{sec: GW invariants} that we refer to an assignment of a non-negative integer to each edge of the $\gpT$-diagram of $Z$ as a skeletal degree. Given such an assignment $\boldsymbol{d}:E(\Gamma) \rightarrow \bZ_{\geq 0}$ we denoted by
\begin{equation}
	\label{eq: GW torus localisation 2}
	\GWdiscarg{Z}{\gpT}{\boldsymbol{d}} \coloneqq \sum_{\substack{\gamma \text{ has skeletal} \\ \text{support }\boldsymbol{d}}} ~~ \int_{[F_{\gamma}]^{\vir}} \frac{1}{e_{\gpT}(N^{\vir}_\gamma)}
\end{equation}
the contribution of all $\gpT$-fixed loci parametrising stable maps whose cover of the torus orbit associated with an edge $e$ has degree $d_e$.

\subsection{Capped localisation}
Each term on the right-hand side of \eqref{eq: GW torus localisation 2} can be written as a product of weights labelled by vertices and edges of the $\gpT$-diagram. Vertices are weighted by Hodge integrals and edges by some closed-form combinatorial factors. The idea of capped localisation is to repackage this decomposition in a way that vertex weights become relative Gromov--Witten invariants of a partial compactification of $\Aaff[5]$ and edge weights become relative invariants of a vector bundle over a rational curve relative to two fibres. The result is a formula expressing the Gromov--Witten invariant with skeletal degree $\boldsymbol{d}$ as a weighted sum over so-called capped markings which record the relative conditions at the interface between partially compactified vertices and edges. To be precise, by a \define{capped marking} $\boldsymbol{\nu}$ we mean the assignment of a partition $\nu_h$ to every half-edge $h$ of the $\gpT$-diagram of $Z$ satisfying
\begin{itemize}
	\item $\nu_h = \varnothing$ whenever $h$ is a leaf;
	
	\item $|\nu_h| = |\nu_{h'}| = d_e$ when $e=(h,h')$.
\end{itemize}
We denote by $\cP_{\Gamma,\boldsymbol{d}}$ the set of all such capped markings. We remark that since we assume $\boldsymbol{d}$ is supported away from \sd\ strata we have $\nu_h = \varnothing$ whenever $h$ is an \sd\ half-edge. This assumption will become crucial later in our proof (see \Cref{rmk: support 1,rmk: support 2}).

In the following we will describe the partial compactifications of vertices and edges and their associated weights needed in order to state the capped localisation formula. We do so by lifting the constructions described in \cite[Sec.~7]{LLLZ09:MathTopVert} and \cite[Sec.~2.3]{MOOP11:GWDTtoric} from the threefold to our fivefold setting.

\subsubsection{Capped vertices}
\label{sec: capped vertex}
We begin by discussing the local model for partial compactifications of vertices in $\Gamma$. For this denote by $U$ the result of blowing up $\bP^1\times\bP^1\times\bP^1$ along the lines
\begin{equation*}
	\infty \times \bP^1 \times 0 \,,\qquad 0 \times \infty \times \bP^1\,,\qquad \bP^1 \times 0 \times \infty 
\end{equation*}
and deleting the pre-image of the three $\Gm[3]$-preserved lines through $(\infty,\infty,\infty)$ under the blow-up morphism. The construction provides us with a partial compactification of $\Aaff[3]\ni (0,0,0)$ where the $i$\textsuperscript{th} coordinate axis is compactified to a $\bP^1$ with normal bundle $N_{\bP^1}U = \cO_{\bP^1} \oplus \cO_{\bP^1}(-1)$. We denote by $D_i\subset U$ the divisor at $\infty\in \bP^1$. Note that $D_1$, $D_2$ and $D_3$ are pairwise disjoint and that whenever the action of a torus $\gpT'\subset \Gm[3]$ on $\Aaff[3]$ is Calabi--Yau, then its action on $D_i \cong \Aaff[2]$ will be so as well.

Now let $v$ be a vertex of the $\gpT$-diagram $\Gamma$ of $Z$. Locally at the associated fixed point $p_v$ the fivefold $Z$ looks like $\Aaff[5] \cong T_{p_v} Z$ with the induced torus action. The choice of a \distdir\ $\epsilon_v$ and cyclic order $\sigma_v = (h_1\, h_2 \, h_3)$ singles out a decomposition
\begin{equation}
	\label{eq: TpvZ decomposition}
	T_{p_v} Z \cong \Aaff[5] = \Aaff[3] \times \Aaff[2] 
\end{equation}
where $\gpT$ acts on the coordinate lines of $\Aaff[2]$ with opposite weights $\epsilon_v$ and $-\epsilon_v$ and the action on $\Aaff[3]$ is Calabi--Yau. We choose the following partial compactification at $p_v$:
\begin{equation*}
	(\widetilde{U}\vertspace \widetilde{D}_{h_1} + \widetilde{D}_{h_2} + \widetilde{D}_{h_3}) \coloneqq (U \times \Aaff[2] \vertspace D_1 \times \Aaff[2] + D_2 \times \Aaff[2] + D_3 \times \Aaff[2])\,.
\end{equation*}
The $\gpT$-action on $\Aaff[3] \times \Aaff[2]$ lifts to $U \times \Aaff[2]$. Its tangent weights at the origin of $D_i \times \Aaff[2] \cong \Aaff[4]$ are
\begin{equation*}
	\epsilon_{\sigma(h_i)}, - \epsilon_{\sigma(h_i)}, \epsilon_v, -\epsilon_v\,.
\end{equation*}

Now suppose we are given a capped marking $\boldsymbol{\nu}$. Then only the partitions $\nu_{h_1}$, $\nu_{h_2}$ and $\nu_{h_3}$ of the half-edges adjacent to $v$ may be non-trivial. We weight the vertex $v$ by the relative Gromov--Witten invariants of $(\widetilde{U}\vertspace \widetilde{D}_{h_1} + \widetilde{D}_{h_2} + \widetilde{D}_{h_3})$:
\begin{equation*}
	\widetilde{C}(v, \boldsymbol{\nu}) \coloneqq \sum_{g\in \bZ} %u^{2g-2+\sum_i \ell(\nu_{h_i})}
	~\int_{[\Mbar{}_{g}^\bullet(\widetilde{U}\vertspace \widetilde{D}_{h_1} + \widetilde{D}_{h_2} + \widetilde{D}_{h_3},\nu_{h_1},\nu_{h_2},\nu_{h_3})]^{\vir}_{\gpT}} 1 \,.
\end{equation*}

\begin{rmk}
	\label{rmk: support 1}
	We want to stress that due to our assumption that $\boldsymbol{d}$ is supported away from \sd\ strata we do not need to compactify vertices in the two \sd\ directions as the associated half-edges are decorated with empty partitions. This fact is going to be crucial for evaluating the vertex weights later in \Cref{cor: vertex weight formula}.
\end{rmk}

\subsubsection{Capped edges}
\label{sec: capped edges}
Let $e$ be an edge with none of its half-edges $h$ and $h'$ being an \sd\ half-edge. At the vertices linked by $e$ we have already chosen a partial compactification whose interface at $h$ and $h'$ is the divisor $\widetilde{D}_{h}\cong \Aaff[4]$ and $\widetilde{D}_{h'}$ respectively. Now let $\gpT$ act on $\bP^1$ with tangent weight $-\epsilon_h$ at $0$ and $-\epsilon_{h'} = \epsilon_h$ at $\infty$. There is a $\gpT$-equivariant rank-four vector bundle $V$ on $\bP^1$ whose fibres at $0$ and $\infty$ are $\widetilde{D}_{h}$ and $\widetilde{D}_{h'}$ compatible with the $\gpT$-actions.

The vector bundle $V$ will serve as the partial compactification for the edge $e$. For a capped marking $\boldsymbol{\nu}$ we weight $e$ by the relative Gromov--Witten invariant
\begin{equation*}
	\widetilde{E}(e,\boldsymbol{\nu}) \coloneqq \sum_{g\in \bZ} %u^{2g-2+\ell(\nu_{h})+\ell(\nu_{h'})} 
	~\int_{[\Mbar{}^\bullet_g(V \vertspace \widetilde{D}_{h} + \widetilde{D}_{h'} , \nu_h , \nu_{h'})]^{\vir}_{\gpT}}1\,.
\end{equation*}

\subsubsection{The formula}
Capped localisation then yields the following formula:
\begin{equation}
	\label{eq: capped localisation}
	\GWdiscarg{Z}{\gpT}{\boldsymbol{d}} = \sum_{\boldsymbol{\nu} \in \cP_{\Gamma,\boldsymbol{d}}} \prod_{e\in E(\Gamma)} \widetilde{E}(e,\boldsymbol{\nu}) \prod_{v\in V(\Gamma)} \widetilde{C}(v,\boldsymbol{\nu}) \prod_{h\in H(\Gamma)}\mf{z}(\nu_h) \left(\epsilon_{\sigma_v(h)} \epsilon_v \right)^{2\ell(\nu_{h})} \,.
\end{equation}
The last product features the gluing terms associated with each interface $\widetilde{D}_h$ where $\ell(\nu)$ denotes the length of a partition $\nu$ and $\mf{z}(\nu) = |\Aut(\nu)| \prod_{i} \nu_i$. As discussed in \cite[Sec.~2.4]{MOOP11:GWDTtoric} and explained in full detail in \cite[Sec.~7]{LLLZ09:MathTopVert} formula \eqref{eq: capped localisation} may be proven by inserting appropriate rubber integrals of $\widetilde{D}_h \times \bP^1$ and combinatorial factors at the half-edges.

\subsection{Simplification of the capped localisation formula}
We will deduce our vertex formula \eqref{eq: vertex formula main part} from \eqref{eq: capped localisation} by explicitly evaluating the vertex and edge weights and reorganising the resulting expression.

\subsubsection{Vertex weights}
We recall the topological vertex formula for the relative Gromov--Witten invariants of the partial compactification $(U\vertspace D_1 + D_2 + D_3)$ of $\Aaff[3]$ described in \Cref{sec: capped vertex}. Let $\gpT$ be a torus acting on $\Aaff[3]$.

\begin{thm}
	\label{thm: vertex formula}
	Let $\nu_1$, $\nu_2$, $\nu_3$ be partitions. If the $\gpT$-action on $\Aaff[3]$ is Calabi--Yau we have
	\begin{equation*}
	\begin{split}
		&\sum_{g} u^{2g-2+\sum_i \ell(\nu_{i})} \int_{[\Mbar{}_{g}^\bullet({U}\vertspace {D}_{1} + {D}_{2} + {D}_{3},\nu_{1},\nu_{2},\nu_{3})]^{\vir}_{\gpT}} 1 \\
		& \qquad\qquad= \left(\prod_{i=1}^3 (-1)^{|\nu_i|} (\ri \epsilon_{i+1})^{-\ell(\nu_{i})}\right) \sum_{\lambda_1,\lambda_2,\lambda_3}  \cW_{\lambda_1,\lambda_2,\lambda_3}(\re^{\ri u}) ~ \prod_{i=1}^3 \frac{\chi_{\lambda_i}(\nu_i)}{\mf{z}(\nu_i)}
	\end{split}
	\end{equation*}
	where $\chi_{\lambda}(\nu)$ is the character of the irreducible representation labelled by a partition $\lambda$ evaluated at the conjugacy class $\nu$.
\end{thm}

This formula was first proven in the case where two partitions are empty \cite{LLZ03:MarinoVafaFormula,OP04:unknot} and later for one partition being empty \cite{LLZ07:2PartnHodgeInt}. The general formula is a consequence of the relative Gromov--Witten/Donaldson--Thomas correspondence for toric threefolds \cite{MOOP11:GWDTtoric} together with the formula for the three-leg Donaldson--Thomas vertex \cite[Eq.~(3.23)]{ORV06:QuCYandCrystals}.

By our assumption that the $\gpT$-action on $Z$ is \locsd\ the above threefold formula will allow us to evaluate our fivefold vertex weights. Indeed, the partial compactification at a vertex $v$ is the product $\widetilde{U} = U \times \Aaff[2]$. A direct comparison of the perfect obstruction theories shows that the relative genus-$g$ invariants of $\widetilde{U}$ differ from those of $U$ by an insertion of $\HodgeLambda[g]{\epsilon_v} \HodgeLambda[g]{-\epsilon_v}$ where
\begin{equation*}
	\HodgeLambda[g]{\epsilon} = \sum_{k\geq 0} \lambda_k (-1)^k \epsilon^{g-k-1}
\end{equation*}
and $\lambda_k$ is the $k$\textsuperscript{th} Chern class of the Hodge bundle. Hence, we get the following corollary as an immediate consequence of \Cref{thm: vertex formula} and Mumford's relation \cite[Sec. 5]{Mu83}
\begin{equation}
	\label{eq: Mumfords relation}
	\HodgeLambda[g]{\epsilon} \, \HodgeLambda[g]{-\epsilon} = (-1)^{g-1} \epsilon^{2g-2}\,.
\end{equation}

\begin{cor}
	\label{cor: vertex weight formula}
	Let $\boldsymbol{\nu}$ be a capped marking. Then
	\begin{flalign*}
		&& \widetilde{C}(v,\boldsymbol{\nu}) = \left(\prod_{i=1}^3 (-1)^{|\nu_i|}(\epsilon_{\sigma(h_i)} \epsilon_v)^{-\ell(\nu_{h_i})}\right) \sum_{\lambda_1,\lambda_2,\lambda_3}  \cW_{\lambda_1,\lambda_2,\lambda_3}(\re^{u \epsilon_v}) ~ \prod_{i=1}^3 \frac{\chi_{\lambda_i}(\nu_i)}{\mf{z}(\nu_i)} \,. && \qed
	\end{flalign*}
\end{cor}

\begin{rmk}
	\label{rmk: loc sd assumption}
	We stress that in order to deduce the corollary it was essential to assume that the torus action on $Z$ is \locsd. Otherwise we could not have used Mumford's relation to reduce the fivefold invariant to a threefold one. In order to establish a vertex formalism which applies beyond the \locsd\ situation a better understanding of quintuple Hodge integrals will be crucial. See \cite{GPS26:5Hodge} for a first step in this direction.
\end{rmk}

\begin{rmk}
	\label{rmk: support 2}
	To be able to evaluate the vertex terms it is also crucial to assume that the skeletal degree is supported away from \sd\ strata. It is unclear to the author whether there is an appropriate substitute for the partial compactification $\widetilde{U}$ we chose at each vertex which lets one drop this assumption. On the level of bare vertices (that is Hodge integrals) dropping the requirement that $d_e=0$ for edges ending at \sd\ half-edges would require closed formulae for the 3-legged threefold vertex with descendants.
\end{rmk}

\subsubsection{Edge weights}
Let $e=(h,h')$ be an edge with none of its half-edges being an \sd\ half-edge. We will evaluate the edge weight $\widetilde{E}(e,\boldsymbol{\nu})$ via virtual localisation. Recall from \Cref{sec: capped edges} that we weight $e$ by the relative Gromov--Witten invariant of a rank-four vector bundle $V$ on $\bP^1$ relative to the fibres at $0$ and $\infty$ which we identified with the divisors $\widetilde{D}_h$ and $\widetilde{D}_{h'}$ respectively. Further, recall that the tangent $\gpT$-weights at the origin of these fibres are
\begin{equation*}
	(\epsilon_{\sigma(h)},-\epsilon_{\sigma(h)},\epsilon_v,-\epsilon_v)\,, \qquad (\epsilon_{\sigma(h')},-\epsilon_{\sigma(h')},\epsilon_{v'},-\epsilon_{v}')\,.
\end{equation*}

Evaluating the relative Gromov--Witten invariants of $(V\vertspace \widetilde{D}_h + \widetilde{D}_{h'})$ via virtual localisation \cite{GP97:virtloc,GV05:RelVirtLoc} yields a decomposition of the edge weight into
\begin{equation}
	\label{eq: edge weight localised}
	\widetilde{E}(e,\boldsymbol{\nu}) = \sum_{\lambda \vdash d_e} A({\nu_h,\lambda}) \cdot B({\lambda}) \cdot  A^{\prime}({\lambda,\nu_{h'}}) \cdot ({\textstyle\prod_i} \lambda_i) \,\mf{z}(\lambda) \, (\epsilon_{\sigma(h)} \epsilon_v \epsilon_{\sigma(h')} \epsilon_{v'})^{2\ell(\lambda)}
\end{equation}
where the middle factor
\begin{equation*}
	B({\lambda}) =\frac{1}{\prod_i \lambda_i} \prod_{i=1}^{\ell(\lambda)} \left(\frac{\lambda_i}{\epsilon_h}\right)^4 \frac{ \Gamma\!\left(\frac{\lambda_i \epsilon_{\sigma_{v'}(h')}}{\epsilon_h}\right) \, \Gamma\!\left(-\frac{\lambda_i \epsilon_{\sigma_{v'}(h')}}{\epsilon_h}\right) \, \Gamma\!\left(\frac{\lambda_i \epsilon_{v'}}{\epsilon_h}\right) \, \Gamma\!\left(-\frac{\lambda_i \epsilon_{v'}}{\epsilon_h}\right) }{ \Gamma\!\left(\frac{\lambda_i \epsilon_{\sigma_{v}(h)}}{\epsilon_h} + 1\right) \, \Gamma\!\left(-\frac{\lambda_i \epsilon_{\sigma_{v}(h)}}{\epsilon_h} + 1\right) \, \Gamma\!\left(\frac{\lambda_i \epsilon_{v}}{\epsilon_h} + 1\right) \, \Gamma\!\left(-\frac{\lambda_i \epsilon_{v}}{\epsilon_h} + 1\right) } 
\end{equation*}
comes from components of a torus fixed stable map which are degree $\lambda_i$ covers of the zero section $\bP^1 \subset V$ fully ramified over $0$ and $\infty$. The $A$-factors are rubber integrals arising from domain components mapping into bubbles at $0$ and $\infty$. More precisely, we have
\begin{equation*}
	A({\nu_h,\lambda}) = \sum_{g} %u^{2g-2+\ell(\nu_h) + \ell(\lambda)} 
	~\int_{[\Mbar{}_{g}^\bullet(\bP^1 , \nu_{h}, \lambda)^\sim]^{\vir}} \frac{\HodgeLambda[g]{\epsilon_{\sigma(h)}} \, \HodgeLambda[g]{-\epsilon_{\sigma(h)}} \, \HodgeLambda[g]{\epsilon_{v}} \, \HodgeLambda[g]{-\epsilon_{v}}}{\epsilon_{h}-\psi_\infty}
\end{equation*}
with a similar formula for $A^{\prime}$. The last factor arises from gluing the domains responsible for the $A$,$B$ and $A'$ terms along the nodes over $0$ and $\infty$. We chose to present the middle term in a rather non-standard way to facilitate the evaluation of this factor.

\begin{lem}
	\label{lem: edge weight B}
	We have
	\begin{equation*}
		B({\lambda}) = (-1)^{p_e \ell(\lambda)+ f_e |\lambda|} \, (\epsilon_{\sigma_v (h)} \epsilon_v \, \epsilon_{\sigma_{v'} (h')} \epsilon_{v'})^{-\ell(\lambda)} \frac{1}{\prod_i \lambda_i}\,.
	\end{equation*}
\end{lem}

\begin{proof}
	By the reflection formula of the Gamma function $\Gamma(z)\Gamma(1-z) =  \frac{\pi}{\sin \pi z}$ we can write $B({\lambda})$ as
	\begin{equation*}
		B({\lambda}) = (\epsilon_{\sigma_v (h)} \epsilon_v \, \epsilon_{\sigma_{v'} (h')} \epsilon_{v'})^{-\ell(\lambda)} \frac{1}{\prod_i \lambda_i} \prod_{i=1}^{\ell(\lambda)} \frac{\sin \frac{\pi \lambda_i \epsilon_{\sigma_v (h)}}{\epsilon_h} \, \sin \frac{\pi \lambda_i \epsilon_{v}}{\epsilon_h} }{\sin \frac{\pi \lambda_i \epsilon_{\sigma_{v'} (h')}}{\epsilon_h} \, \sin \frac{\pi \lambda_i \epsilon_{v'}}{\epsilon_h} } \,.
	\end{equation*}
	Now due to the linear relations \eqref{eq: weight reln 1} or \eqref{eq: weight reln 2} satisfied by the weights $\epsilon_h$, $\epsilon_{\sigma_v (h)}$, $\epsilon_v$, $\epsilon_{\sigma_{v'} (h')}$ and $\epsilon_{v'}$ the last product simplifies to a factor $\pm 1$. One can check in a case-by-case analysis that by our definition of the \frmxp\ we have
	\begin{equation*}
		\frac{\sin \frac{\pi \lambda_i \epsilon_{\sigma_v (h)}}{\epsilon_h} \, \sin \frac{\pi \lambda_i \epsilon_{v}}{\epsilon_h} }{\sin \frac{\pi \lambda_i \epsilon_{\sigma_{v'} (h')}}{\epsilon_h} \, \sin \frac{\pi \lambda_i \epsilon_{v'}}{\epsilon_h} } = (-1)^{p_e + f_e \lambda_i}\,. \qedhere
	\end{equation*}
\end{proof}

The $A$-terms are evaluated as follows.

\begin{lem}
	\label{lem: edge weight A}
	We have
	\begin{equation*}
		A({\nu_h,\lambda}) = (\epsilon_{\sigma_v (h)} \epsilon_v)^{-\ell(\nu_h)-\ell(\lambda)} \sum_{\mu} \exp \left( \frac{\kappa(\mu)}{2} \frac{\epsilon_{\sigma_v (h)} \epsilon_v}{\epsilon_h} \right) ~ \frac{\chi_\mu (\nu_h)}{\mf{z}(\nu_h)} \frac{\chi_\mu (\lambda)}{\mf{z}(\lambda)} 
	\end{equation*}
\end{lem}

\begin{proof}
	The claim follows from Mumford's relation \eqref{eq: Mumfords relation} together with the rubber integral formula
	\begin{equation*}
	\begin{split}
		\sum_g u^{2g-2+\ell(\nu_1)+\ell(\nu_2)} & \int_{[\Mbar{}_{g}^\bullet(\bP^1 , \nu_1, \nu_2)^\sim]^{\vir}} \frac{1}{\epsilon - \psi_{\infty}} \\
		& \qquad = \sum_{\mu} \exp \left( \frac{\kappa(\mu)}{2} \frac{u}{\epsilon} \right) ~ \frac{\chi_\mu (\nu_1)}{\mf{z}(\nu_1)} \frac{\chi_\mu (\nu_2)}{\mf{z}(\nu_2)}\,.
	\end{split}
	\end{equation*}
	This equation is proven in \cite[Prop.~5.4 \& Eq.~(17)]{LLZ07:2PartnHodgeInt} by relating the rubber integral to double Hurwitz numbers.
\end{proof}

If we insert the formulae from \Cref{lem: edge weight B} and \labelcref{lem: edge weight A} into equation \eqref{eq: edge weight localised} and employ the orthogonality relations
\begin{align}
	\sum_\lambda \frac{\chi_{\mu_1} (\lambda) \, \chi_{\mu_2} (\lambda)}{\mf{z}(\lambda)}  & = \delta_{\mu_1,\mu_2}\,, \label{eq: chi orthogonality relation} \\
	\sum_\lambda (-1)^{\ell(\lambda)} \frac{\chi_{\mu_1} (\lambda) \, \chi_{\mu_2} (\lambda)}{\mf{z}(\lambda)} & =  \delta_{\mu_1,\mu_2^{\transpose}} (-1)^{|\mu_1|} \,. \nonumber \qedhere
\end{align}
we arrive at the following expression for the edge weights.

\begin{cor}
	\label{lem: edge weight formula}
	Let $\boldsymbol{\nu}$ be a capped marking. Then
	\begin{equation*}
		\begin{split}
			\widetilde{E}(e,\boldsymbol{\nu}) & = (-1)^{(f_e+p_e)|\nu_h|} (\epsilon_{\sigma_v (h)} \epsilon_v)^{-\ell(\nu_h)} (\epsilon_{\sigma_{v'} (h')} \epsilon_{v'})^{-\ell(\nu_{h'})} \\[0.5em]
			& \qquad \times \sum_{\mu_h,\mu_{h'}} \exp \left( \frac{\kappa(\mu_h)}{2} \frac{ \epsilon_{\sigma_v (h)} \epsilon_v + (-1)^{p_e+1}\epsilon_{\sigma_{v'} (h')} \epsilon_{v'}}{\epsilon_h} \right) ~ \frac{\chi_{\mu_h} (\nu_h)}{\mf{z}(\nu_h)} \frac{\chi_{\mu_{h'}} (\nu_{h'})}{\mf{z}(\nu_{h'})}
		\end{split}
	\end{equation*}
	where the sum is over tuples of partitions $(\mu_h,\mu_{h'})$ satisfying $\mu_h = \mu_{h'}$ if the \mxp\ of $e$ is even and $\mu_h = \mu_{h'}^{\transpose}$ if it is odd. \hfill \qed
\end{cor}
 
\begin{rmk}
	Both in \Cref{lem: edge weight B} and \labelcref{lem: edge weight A} the assumptions that the $\gpT$-action is Calabi--Yau and \locsd\ are used crucially: First, the assumptions lead to the collapse of the $B$-term and second, they enable us to use Mumford's relation which is key to evaluate the $A$-terms. To go beyond \locsd\ torus actions new ideas will be necessary.
\end{rmk}

\subsubsection{Putting everything together}
If we plug our formulae for the vertex and edge weights (\Cref{cor: vertex weight formula} and \labelcref{lem: edge weight formula}) into the capped localisation formula \eqref{eq: capped localisation} we obtain
\begin{equation*}
\begin{split}
	&\qquad \GWdiscarg{Z}{\gpT}{\boldsymbol{d}} = \\
	&\sum_{\boldsymbol{\nu}\in \cP_{\Gamma,\boldsymbol{d}}} \sum_{\boldsymbol{\lambda}}  \sum_{\boldsymbol{\mu}\in \cP_{\Gamma, \boldsymbol{p}, \boldsymbol{d}}} ~  \prod_{e=(h,h')} (-1)^{(f_e+p_e) |\mu_h|} \exp\left(  \frac{\kappa(\mu_h)}{2} \frac{\epsilon_{\sigma_v (h)} \epsilon_v + (-1)^{p_e+1}\epsilon_{\sigma_{v'} (h')} \epsilon_{v'}}{\epsilon_h} \right) \\ 
	& \hspace{6.6em}\times \prod_v  \cW_{\lambda_{h_1},\lambda_{h_2},\lambda_{h_3}}\!(\re^{u \epsilon_v}) ~ \prod_h \frac{\chi_{\lambda_{h}}(\nu_h) \, \chi_{\mu_h}(\nu_h) }{\mf{z}(\nu_h)}
\end{split}
\end{equation*}
where the second sum runs over tuples of partitions $\boldsymbol{\lambda}=(\lambda_{h_1},\lambda_{h_2},\lambda_{h_3})_{v\in V(\Gamma)}$ assigning a partition to each of the three half-edges permuted by the cyclic order $\sigma_v=(h_1\, h_2\, h_3)$ at a vertex $v$. We can carry out the sum over $\boldsymbol{\nu}$ using the orthogonality relation \eqref{eq: chi orthogonality relation}. The result of this manipulation is precisely the vertex formula stated in \Cref{thm: vertex formalism main part}.

%-------------------------------------------------------------------------------
\appendix
%-------------------------------------------------------------------------------
\section{The plethystic logarithm in localised K-theory}
\label{sec: PLog}

\subsection{Integrality}
Let $\gpT$ be a torus. We consider the representation ring of $\gpT$ with coefficients in $\bZ$ localised at the augmentation ideal, i.e.~the ideal of zero dimensional virtual representations:
\begin{equation*}
	R \coloneqq \Rep(\gpT)_{\mr{loc}} = K_{\gpT}(\mr{pt})_{\mr{loc}}\,.
\end{equation*}
If we fix an isomorphism $\gpT \cong \Gm[m]$ we get the following presentation:
\begin{equation*}
	R \cong \bZ\left[q_1^{\pm 1},\ldots,q_m^{\pm 1},\left\{\frac{1}{\prod_{i=1}^m(1-q_i^{n_i})}\right\}_{\boldsymbol{n}\in\bZ^m\setminus\{0\}}\right]\,.
\end{equation*}
The ring $R$ is a lambda ring. The $k$\textsuperscript{th} Adams operation $\Psi_k$ acts on one-dimensional representations $q$ as $\Psi_k(q)= q^k$. We get a lambda ring structure on $R\llbracket Q_1,\ldots,Q_\ell\rrbracket$ by declaring that it acts on monomials as $\Psi_k(r Q^{\boldsymbol{d}}) \coloneqq \Psi_k(r) Q^{k\boldsymbol{d}}$. Let $I\subset R$ be the ideal generated by $Q_1,\ldots,Q_\ell$. We define the plethystic logarithm as
\begin{equation*}
\begin{tikzcd}
	\Log: &[-3em] 1 + I \cdot R\llbracket Q_1,\ldots,Q_\ell\rrbracket \ar[r] & I \cdot R\llbracket Q_1,\ldots,Q_\ell\rrbracket \otimes \bQ\,, \\[-2em]
	& G \ar[maps to,r] & {\displaystyle\sum_{k>0}}\frac{\mu(k)}{k} \, \Psi_k(\log G)\,.
\end{tikzcd}
\end{equation*}
A priori, by the above definition we should expect the image $\Log(G)$ of a power series with coefficients in $R$ to feature coefficients in $R\otimes \bQ$ --- one source of denominators being the factor $1/k$ and the other being the logarithm. In contrast to this expectation, one can show that the plethystic logarithm preserves integrality.

\begin{lem}
	\label{lem: PLog preserves integrality}
	The image of the plethystic logarithm lies in $I \cdot R\llbracket Q_1,\ldots,Q_\ell\rrbracket$.
\end{lem}

\begin{rmk}
	The analogous statement for $\Rep(\gpT)$, that is without localising at the augmentation ideal, holds true since the plethystic exponential acts on a representation $V$ times $Q^{\boldsymbol{d}}$ as $\Exp(V Q^{\boldsymbol{d}}) = \sum_{n\geq 0} \mr{Sym}_n(V) Q^{n\boldsymbol{d}}$ meaning that it preserves integrality. Thus, the same must be true for the plethystic logarithm since the relation $\Exp(\Log G) = G$ allows to solve for $\Log G$ recursively. Hence, the insight of \Cref{lem: PLog preserves integrality} is that the feature of preserving integrality persists after localising at the augmentation ideal.
\end{rmk}

\subsection{The proof of \Cref{lem: PLog preserves integrality}}
To prove \Cref{lem: PLog preserves integrality} we follow Konishi \cite[Sec.~5]{Kon06:Integrality} generalising ideas of Peng \cite{Pen07:ProofGVConj}. Let $G \in 1 + I \cdot R\llbracket Q_1,\ldots,Q_\ell\rrbracket$. We write $G_{\boldsymbol{d}}$ for the coefficients of this power series and $\Omega_{\boldsymbol{d}}$ for the coefficients of its plethystic logarithm:
\begin{equation*}
	\Log\left(1+\sum_{\boldsymbol{d} \neq 0} Q^{\boldsymbol{d}} ~ G_{\boldsymbol{d}}\right) \coloneqq \Log(G) \eqqcolon \sum_{\boldsymbol{d} \neq 0} Q^{\boldsymbol{d}} ~ \Omega_{\boldsymbol{d}} \,.
\end{equation*}
We need to show that $\Omega_{\boldsymbol{d}} \in R \subset R\otimes \bQ$. To do so we first find a formula for these coefficients by expanding the left-hand side of the above equation. Following \cite[Sec.~5.1]{Kon06:Integrality}, we write
\begin{equation*}
	D(\boldsymbol{d}) \coloneqq \big\{~\boldsymbol{\delta}\in\bZ^{\ell}_{\geq 0}\setminus\{0\} ~ \big| ~ \delta_i \leq d_i \text{ for all }i ~\big\}
\end{equation*}
and call an element $\boldsymbol{n} \in \bZ_{\geq 0}^{D(\boldsymbol{d})}$ a multiplicity of $\boldsymbol{d}$ if
\begin{equation*}
	\sum_{\boldsymbol{\delta}\in D(\boldsymbol{d})} n_{\boldsymbol{\delta}}  \boldsymbol{\delta}  = \boldsymbol{d}\,.
\end{equation*}
With this notation we have\footnote{We remark that there is a small typo in \cite{Kon06:Integrality} in the formula in Lemma 5.2 and the one for $G^\Gamma_{\vec{d}}(q)$ stated before the lemma: The denominator $k'|n|$ should not be present.}
\begin{equation}
	\label{eq: formula Omega}
	\Omega_{\boldsymbol{d}} = -\sum_{k \vertspace \gcd(\boldsymbol{d})} \sum_{\boldsymbol{n}} \frac{1}{k|\boldsymbol{n}|} \sum_{k'|k} \mu\left(\frac{k}{k'}\right) \frac{(k'|\boldsymbol{n}|)!}{\prod_{\boldsymbol{\delta}\in D(\boldsymbol{d})} (k'n_{\boldsymbol{\delta}})!}  \left((-1)^{|\boldsymbol{n}|}\prod_{\boldsymbol{\delta}\in D(\boldsymbol{d})} \Psi_{k/k'}\big(G_{\boldsymbol{\delta}}^{n_{\boldsymbol{\delta}}}\big)\right)^{k'}\,.
\end{equation}
where $\mu$ is the M\"{o}bius function and the second sum runs over all multiplicities $\boldsymbol{n}$ of $\boldsymbol{d}/k$ satisfying $\gcd(\boldsymbol{n}) = 1$. To prove \Cref{lem: PLog preserves integrality} it therefore suffices to show the following.

\begin{prop}
	\label{prop: every term is integral}
	For $G\in R$, $\boldsymbol{n} \in \bZ^\ell_{>0}$ with $\gcd(\boldsymbol{n}) = 1$ and $k > 0$ we have
	\begin{equation*}
		\frac{1}{k|\boldsymbol{n}|}  \sum_{k'|k} \mu\left(\frac{k}{k'}\right) \frac{(k'|\boldsymbol{n}|)!}{\prod_{i} (k'n_i)!}  \left(\Psi_{k/k'} G \right)^{k'} \in R\,.
	\end{equation*}
\end{prop}

We require the following two lemmas.

\begin{lem}
	\label{lem: Konishis lemma}
	\cite[Lem.~A.2]{Kon06:Integrality} Let $\boldsymbol{n}\in \bZ_{>0}^\ell$ with $\gcd(\boldsymbol{n})=1$.
	\begin{enumerate}
		\item For $k\in\bZ_{>0}$ we have
		\begin{equation*}
			\frac{(k|\boldsymbol{n}|)!}{\prod_i(k n_i)!} \equiv 0 \mod |\boldsymbol{n}|\,.				
		\end{equation*}
		
		\item For $p$ a prime, $a\in\bZ_{> 0}$ and $k\in\bZ_{>0}$ not divisible by $p$ we have
		\begin{flalign*}
			&& \frac{(p^{a} k|\boldsymbol{n}|)!}{\prod_i (p^{a} k n_i)!} \equiv \frac{(p^{a-1} k|\boldsymbol{n}|)!}{\prod_i (p^{a-1} k n_i)!} \mod p^a |\boldsymbol{n}| \,. && \qed
		\end{flalign*}
	\end{enumerate}
\end{lem}

\begin{lem}
	\label{lem: Gpi vs Psi Gpi minus 1}
	Let $p$ be a prime and $a\in\bZ_{> 0}$ be an integer. Then for all $G\in R$ we have
	\begin{equation*}
		G^{p^a} - \left(\Psi_p G\right)^{p^{a-1}} \equiv 0 \mod p^a \,.
	\end{equation*}
\end{lem}
\begin{proof}
	We first prove the relation assuming $G\in \Rep(\gpT) \subset R$. We will prove the claim by induction on $a$. For $a=1$ the claim holds by the standard argument proving additivity of the Frobenius morphism. Now suppose the claim holds for $a-1$, that is there exists a $g\in R$ such that
	\begin{equation*}
		G^{p^{a-1}} - \left(\Psi_p G\right)^{p^{a-2}} = p^{a-1} g\,.
	\end{equation*}
	We find that
	\begin{align*}
		G^{p^a} - \left(\Psi_p G\right)^{p^{a-1}} & = \left( \left(\Psi_p G\right)^{p^{a-2}} + p^{a-1} g \right)^p - \left(\left(\Psi_p G\right)^{p^{a-2}}\right)^p \\
		& = \sum_{k>0} \binom{p}{k}  \, p^{(a-1)k} g^k \, \left(\Psi_p G\right)^{p^{a-2}(p-k)}\\
		& \equiv 0 \mod p^a
	\end{align*}
	since $p$ divides $\binom{p}{k}$ for $k>0$. This means the claim is proven for $G\in \Rep(\gpT)$.
	
	We now extend to general $G \in R$ by reduction to the last case. For this we express $G= {G_1}/{G_2}$ as a ratio of elements $G_1\in \Rep(\gpT)$ and $G_2 = \prod_{j} (1-\chi_j)$ where each $\chi_j$ is a monomial in $\Rep(\gpT)$. We can write 
	\begin{equation*}
		G^{p^a} - \left(\Psi_p G\right)^{p^{a-1}} = \frac{1}{G_2^{p^a} (\Psi_p G_2)^{p^{a-1}}} \left[ G_1^{p^a} \left((\Psi_p G_2)^{p^{a-1}} - G_2^{p^{a}}\right) + G_2^{p^a} \left(G_1^{p^a} - (\Psi_p G_1)^{p^{a-1}}\right)  \right] \,.
	\end{equation*}
	Since $G_1,G_2\in \Rep(\gpT)$ we know that the numerator on the right-hand side is divisible by $p^a$ and so the \namecref{lem: Gpi vs Psi Gpi minus 1} follows.
\end{proof}

\begin{rmk}
	The proof of \Cref{lem: Gpi vs Psi Gpi minus 1} is the only place in our proof of \Cref{lem: PLog preserves integrality} that uses any features special to $R = \Rep(\gpT)_{\mr{loc}}$. All other steps in our proof hold for an arbitrary lambda ring. Thus, the conclusion of \Cref{lem: PLog preserves integrality} actually holds for any lambda ring $R$ satisfying \Cref{lem: Gpi vs Psi Gpi minus 1}.
\end{rmk}

Now we are equipped to prove \Cref{prop: every term is integral} from which \Cref{lem: PLog preserves integrality} follows by equation \eqref{eq: formula Omega}.

\begin{proof}[Proof of \Cref{prop: every term is integral}]
	Our proof closely follows Konishi  \cite[Sec.~A.2]{Kon06:Integrality}. First, note that for $k=1$ the claim of the \namecref{prop: every term is integral} holds trivially. Thus, assume $k>1$ and write $k=\prod_{j=1}^s p_{\mathrlap{j}}{\mathstrut}^{a_j}$ for its prime decomposition. It suffices to show that
	\begin{equation}
		\label{eq: term Omega after prime decomposition}
		\sum_{k'|k} \mu\left(\frac{k}{k'}\right) \frac{(k'|\boldsymbol{n}|)!}{\prod_{i} (k'n_i)!}  \left(\Psi_{k/k'} G \right)^{k'} 
	\end{equation}
	is an element of $R$ which is divisible by $p_{\mathrlap{j}}{\mathstrut}^{a_j} |\boldsymbol{n}|$ for all $j\in\{1,\ldots,s\}$. Without loss of generality we may prove the claim for $j=1$. Writing $l\coloneqq k/p_{\mathrlap{1}}{\mathstrut}^{a_1}$ observe that since for every divisor $k'$ of $k$ we have
	\begin{equation*}
		\mu\left(\frac{k}{k'}\right) = \begin{cases}
			\mu\left(\frac{l}{l'}\right) & k'=p_1^{a_1}l'\\
			-\mu\left(\frac{l}{l'}\right) & k'=p_1^{a_1-1}l'\\
			0 & \text{otherwise}\\
		\end{cases}
	\end{equation*}
	expression \eqref{eq: term Omega after prime decomposition} equates to
	\begin{align*}
		\sum_{l'\vertspace l} \mu\left(\frac{l}{l'}\right)~ & \left( \frac{(p_1^{a_1} l'|\boldsymbol{n}|)!}{\prod_i (p_1^{a_1} l'n_i)!} \left(\Psi_{l/l'}G\right)^{p_1^{a_1}l'} - \frac{(p_1^{a_1-1} l'|\boldsymbol{n}|)!}{\prod_i (p_1^{a_1-1} l'n_i)!} \left(\Psi_{p_1 l/l'}G\right)^{p_1^{a_1-1}l'}\right) \\[0.8em]
		= & \sum_{l'\vertspace l} \mu\left(\frac{l}{l'}\right) ~ \Psi_{l/l'}(G)^{p_1^{a_1}l'} \underbracket{\left(\frac{(p_1^{a_1} l'|\boldsymbol{n}|)!}{\prod_i (p_1^{a_1} l'n_i)!} - \frac{(p_1^{a_1-1} l'|\boldsymbol{n}|)!}{\prod_i (p_1^{a_1-1} l'n_i)!}\right)}_{(\star)} \\
		+ & \sum_{l'\vertspace l} \mu\left(\frac{l}{l'}\right) ~ \underbracket{\frac{(p_1^{a_1-1} l'|\boldsymbol{n}|)!}{\prod_i (p_1^{a_1-1} l'n_i)!}}_{(*)} \underbracket{\left( \Psi_{l/l'}(G^{l'})^{p_1^{a_1}} - \Psi_{p_1 l/l'}(G^{l'})^{p_1^{a_1-1}}\right)\vphantom{\frac{(p_1^{a_1-1} l'|\boldsymbol{n}|)!}{\prod_i (p_1^{a_1-1} l'n_i)!})}}_{(\dagger)}\,.
	\end{align*}
	The \namecref{prop: every term is integral} then follows from the facts that $p_1^{a_1}|\boldsymbol{n}|$ divides the term ($\star$) by \Cref{lem: Konishis lemma} (ii) and that $|\boldsymbol{n}|$ divides ($*$) by \Cref{lem: Konishis lemma} (i). Finally, $p_1^{a_1}$ divides the term ($\dagger$) by \Cref{lem: Gpi vs Psi Gpi minus 1}.
\end{proof}

\begingroup
\setlength{\emergencystretch}{.5em}
\renewcommand*{\bibfont}{\footnotesize}
\printbibliography\medskip
\endgroup

\end{document}